\documentclass[11pt]{amsart}
\usepackage{enumitem, amssymb}
\usepackage{fontawesome}
\usepackage{tikz}
\usetikzlibrary{decorations.markings}

\numberwithin{equation}{section}

\DeclareMathOperator{\OCAT}{OCA_T}

\DeclareMathOperator{\SL}{SL}

\DeclareMathOperator{\Aut}{Aut}

\DeclareMathOperator{\dist}{dist}

\newcounter{my_enumerate_counter}

\newcommand{\bbC}{{\mathbb C}}

\newcommand{\bbN}{{\mathbb N}}

\newcommand{\cU}{{\mathcal U}}
\newcommand{\cZ}{{\mathcal Z}}

\newcommand{\Ext}{{\rm Ext}}

\newcommand{\wst}{\mathrm{W}^*}
\newcommand{\cstar}{$\mathrm{C}^*$}
\newcommand{\cst}{\mathrm{C}^*}
\newcommand{\cstr}{\mathrm{C}^*_r}

\DeclareMathOperator{\Ad}{Ad}

\DeclareMathOperator{\rank}{rank}

\newcommand{\cO}{\mathcal O}

\newcommand{\bbZ}{{\mathbb Z}}

\newcommand{\cW}{{\mathcal W}}

\newcommand{\rs}{\restriction}

\newcommand{\cB}{\mathcal B}
\newcommand{\cK}{\mathcal K}
\newcommand{\cM}{\mathcal M}
\newcommand{\cQ}{\mathcal Q}

\newcommand{\calD}{\mathcal D}

\newcommand{\toPsi}{\buildrel\Psi\over \to}

\makeatletter
\newcommand{\customlabel}[2]{%
\protected@write \@auxout {}{\string \newlabel {#1}{{#2}{\thepage}{}{}{}}}}
\makeatother

\newcommand{\oneHS}[1]{\|#1\|_{\textrm{HS},1}}
\newcommand{\hHS}[1]{\|#1\|_{\textrm{HS},H}}
\newcommand{\kHS}[1]{\|#1\|_{\textrm{HS},K}}

\newcommand{\nHS}[1]{\|#1\|_{\textrm{HS},n}}

\newtheorem{theorem}{Theorem}[section]
\newtheorem*{theorem*}{Theorem}
\newtheorem{proposition}[theorem]{Proposition}

\newtheorem*{proposition*}{Proposition}
\newtheorem{lemma}[theorem]{Lemma}
\newtheorem*{lemma*}{Lemma}
\newtheorem{corollary}[theorem]{Corollary}
\newtheorem*{corollary*}{Corollar}

\newtheorem*{fact*}{Fact}
\newtheorem{convention}[theorem]{Convention}

\theoremstyle{definition}
\newtheorem{definition}[theorem]{Definition}
\newtheorem*{definition*}{Definition}
\newtheorem{claim}[theorem]{Claim}
\newtheorem*{claim*}{Claim}

\newtheorem{conjecture}[theorem]{Conjecture}
\newtheorem*{conjecture*}{Conjecture}

\newtheorem{theoremi}{Theorem}
\newtheorem{corollaryi}[theoremi]{Corollary}

\theoremstyle{remark}

\newtheorem*{example*}{Example}

\newtheorem*{remark*}{Remark}

\newtheorem*{note*}{Note}
\newtheorem*{question*}{Question}

\author{Ilijas Farah}
\thanks{Partially supported by NSERC.}
\address{Department of Mathematics and Statistics,
York University,
4700 Keele Street,
Toronto, Ontario, Canada, M3J
1P3} 
\address{Matemati\v cki Institut SANU\\
Kneza Mihaila 36\\
11\,000 Beograd, p.p. 367\\
Serbia}
\email{email: ifarah@yorku.ca}
\urladdr{https://ifarah.mathstats.yorku.ca/}

\thanks{ORCID iD https://orcid.org/0000-0001-7703-6931}

\title[Calkin etc.]{The Calkin algebra, Kazhdan's property (T), and strongly self-absorbing \cstar-algebras}
\date{\today}
\begin{document}
\maketitle

\begin{abstract}  It is well-known that the relative commutant of every separable nuclear \cstar-sub\-al\-geb\-ra of the Calkin algebra has a unital copy of Cuntz algebra $\cO_\infty$.  We prove that the Calkin algebra has a separable \cstar-subalgebra whose relative commutant has no simple, unital, and noncommutative \cstar-subalgebra.
On the other hand, the corona  of every stable, separable \cstar-algebra that tensorially absorbs the Jiang--Su algebra $\cZ$ has the property that the relative commutant of every separable \cstar-subalgebra contains a unital copy of $\cZ$. Analogous result holds for other strongly self-absorbing \cstar-algebras.   
As an application, the Calkin algebra is not isomorphic to the corona of the stabilization of the Cuntz algebra~$\cO_\infty$,  any other Kirchberg algebra, or even the corona of the stabilization of any unital, $\cZ$-stable \cstar-algebra. 
\end{abstract}

The most intriguing open problem about the Calkin algebra $\cQ(H)$ is whether it can have a K-theory reversing automorphism (\cite{BrDoFi:Unitary}, \cite{BrDoFi}). A possible line of  attack on this problem was suggested by Shuang Zhang. The corona $\cQ(\cO_\infty\otimes \cK)$ of the stabilization of the Cuntz algebra $\cO_\infty$ is K-theoretically indistiguishable from $\cQ(H)$, but it is known that the Kirchberg--Phillips classification theorem 
(see \cite{Phi:Classification}, \cite{Ror:Classification}, \cite{gabe2019classification}) implies that it has a $K$-theory reversing automorphism, and an isomorphism between $\cQ(H)$ and $\cQ(\cO_\infty\otimes \cK)$ would give  a K-theory reversing automorphism of $\cQ(H)$.

Remarkably, many of the defining properties of $\cQ(H)$  are shared by $\cQ(\cO_\infty\otimes\cK)$ and coronas of other separable, purely infinite, non-unital \cstar-algebra. These algebras have real rank zero and satisfy a (generalized) Weyl--von Neumann theorem (\cite{zhang1992certain}; for other coronas with real rank zero see \cite[Theorem~3.8 and definitions on pp. 245--246]{ng2022real}). 
They are simple and purely infinite:  By \cite[Theorem~3.2]{rordam1991ideals} or by  \cite[Definition~2.5 and Theorem~2.8]{lin1991simple} and \cite{lin2004simple}, if $A$ is unital then $\cQ(A\otimes \cK)$ is simple if and only if $A=M_n(\bbC)$ or $A$ is simple and purely infinite.  Every separable, unital \cstar-subalgebra of $\cQ(H)$ is equal to its double commutant (this is a consequence of Voiculescu's theorem for $\cQ(H)$, and proved in  \cite[Theorem~B]{giordano2019relative} using the Elliott--Kucerovsky theory of absorbing extensions,~\cite{elliott2007relative} for $\cQ(\cO_\infty\otimes \cK)$; see also \cite{ng2018double}). Many general properties of the Ext functor resemble the ones familiar from the BDF theory (\cite{BrDoFi:Unitary},  \cite{BrDoFi}, also  \cite{Dav:C*}) by  \cite{ng2019functorial}. Nevertheless\dots

\begin{theoremi}\label{T.A}
	The Calkin algebra is not isomorphic to the corona of the stabilization of  $\cO_\infty$. 
\end{theoremi}

To put this theorem in proper context, see the beginning of \S\ref{S.concluding}. Its proof does not use set theory at all, and model theory is making only a cameo appearance in Corollary~\ref{C2} (not used in the proof of our main results).

The proof of Theorem~\ref{T.A} breaks down as follows. Theorem~\ref{T.B} below is based on an extension of S. Wasserman's proof that a certain quotient of the full group algebra associated with a property (T) group with infinitely many inequivalent irreducible representations on finite-dimensional Hilbert spaces has the property that its $\Ext$ is not a group  (\cite{wassermann1991c}, rediscovered in \cite[Theorem~3.7.10]{higson2000analytic}). More precisely, we strengthen the variants of Wassermann's  result developed in \cite [Corollary~6.5.11]{brown2006invariant} and \cite[\S 17]{BrOz:C*}.  The following is proved at the end of~\S\ref{S.Gamma.A}, as a consequence of Proposition~\ref{P.linfty}. 

\begin{theoremi}\label{T.B} There is a separable \cstar-subalgebra $B$ of $\cQ(H)$ whose relative commutant  $\cQ(H)\cap B'$ has no  simple, noncommutative \cstar-subalgebra that contains the unit of $\cQ(H)$.  
\end{theoremi}

The emphasis in the conclusion is of course on the fact that the simple \cstar-subalgebra is assumed to be a unital subalgebra of $\cQ(H)$.   This is because every projection $p$ in $\cQ(H)\cap B'$ corresponds (after identifying $\cQ(H)$ with its corner corresponding to $p$) to an extension\footnote{As customary, we identify extensions with their Busby invariants, that is, unital $^*$-homomorphisms into the Calkin algebra. See e.g., \cite{higson2000analytic}.}  of $B$, defined by $\Phi(b):=pb$,  and equivalent extensions correspond to projections that are Murray--von Neumann equivalent in $\cQ(H)\cap B'$. Finally, by Voiculescu's theorem (see e.g., \cite{Arv:Notes} or \cite{BrOz:C*}) every extension of $B$ absorbs the trivial extension. Therefore some $p_0\in\cQ(H)\cap B’$ corresponds to the trivial extension. By repeatedly applying Voiculescu's theorem, one finds orthogonal projections $p_n$, for $n\geq 1$, in $p_0(\cQ(H)\cap B’)$, each one of them Murray—von Neumann equivalent to~$p_0$. If $v_n$ is a partial isometry in $\cQ(H)\cap B’$ such that $v_n^* v_n=p_0$ and $v_nv_n^*=p_n$, then $\cst(v_n : n\in \bbN)$  is a \cstar-subalgebra of $\cQ(H)\cap B’$ isomorphic to $\cO_\infty$.

The other component of the proof of Theorem~\ref{T.A} applies to coronas of stablizations of strongly self-absorbing \cstar-algebras (see \S\ref{S.Terminology} below). The following is Theorem~\ref{T.Q(A)}.

\begin{theoremi}\label{T.C} Suppose that $D$ is a separable, unital, and  strongly self-absorbing  \cstar-algebra. If $A$ is separable, unital, and tensorially $D$-absorbing then  for every separable \cstar-subalgebra $B$ of $\cQ(A\otimes \cK)$ there is an isomorphic copy of $D$ in $\cQ(A\otimes \cK)\cap B'$ that contains the unit of $\cQ(A\otimes \cK)$. 
\end{theoremi}

Theorem~\ref{T.B} and Theorem~\ref{T.C} immediately imply the following strengthening of Theorem~\ref{T.A}. Here $\cZ$ denotes the Jiang--Su algebra,  the initial object in the category of strongly self-absorbing \cstar-algebras (\cite{winter2011strongly}), and for Elliott-classifiable \cstar-algebras see \cite{Ror:Classification}.  
 
 \begin{corollaryi}\label{C.D}
	The Calkin algebra is not isomorphic to the corona of $A\otimes \cK$, for any separable, unital, and tensorially $\cZ$-absorbing  \cstar-algebra~$A$. 
		In particular, if~$A$ is any Elliott-classifiable \cstar-algebra  then~$\cQ(A\otimes \cK)$ is not isomorphic to $\cQ(H)$. \qed 
\end{corollaryi}

It is well-known that many instances of Corollary~\ref{C.D} can be proven by K-theoretic means, as follows.  To the short exact sequence
\[
0\to A\otimes\cK \to \cM(A\otimes \cK)\to \cQ(A\otimes \cK)\to 0
\]
one associates the six-term exact sequence

\begin{tikzpicture}
	\matrix[row sep=0.6cm,column sep=.6cm] {
		\node  (AK0) {$K_0(A\otimes\cK)$}; &
		\node (M0)  {$K_0(\cM(A\otimes \cK))$};&
		\node (Q0)  {$K_0(\cQ(A\otimes \cK))$}; \\ 
		\node (Q1) {$K_1(\cQ(A\otimes \cK))$};&
		\node (M1)  {$K_1(\cM(A\otimes \cK))$};&
		\node  (AK1) {$K_1(A\otimes\cK)$}; \\
	};
	\draw (AK0) edge [->] (M0);
	\draw  (M0) edge[->] (Q0); 
	\draw  (Q0) edge[->] (AK1); 
	\draw  (AK1) edge[->] (M1); 
	\draw  (M1) edge[->] (Q1); 
	\draw  (Q1) edge[->] (AK0); 
\end{tikzpicture}.

Since the K-theory of the stable multiplier algebra $\cM(A\otimes \cK)$ (\cite[Definition~12.1.3]{blackadar1998k}) is trivial (\cite[Proposition~12.2.1]{blackadar1998k}),  the vertical arrows in this six-term exact sequence  are isomorphisms.  Therefore $K_i(A)=K_{1-i}(\cQ(A\otimes \cK))$ (this is \cite[Corollary~12.2.3]{blackadar1998k}). This implies many instances of Corollary~\ref{C.D}, but not Theorem~\ref{T.A} because $K_0(\cO_\infty)=\bbZ=K_1(\cQ(H))$ and $K_1(\cO_\infty)=0=K_0(\cQ(H))$. 

 For a unital \cstar-algebra $A$ consider the \emph{path algebra}\footnote{There is no universally accepted notation for the path algebra. Some authors use~$A_\infty$, a notation largely established for the more common asymptotic sequence algebra $\ell_\infty(A)/c_0(A)$.  The authors of \cite{guentner2000equivariant} used somewhat cumbersome ${\mathfrak A}A={\mathfrak T}A/{\mathfrak T}_0 A$ and the original notation in \cite{connes1989almost}, $\cQ(A)$, is nowadays commonly used for coronas. The terminology `path algebra’ has been suggested by G\' abor Szab\' o.} 
\[
A_{\text{\faForward}}:=C_b([0,\infty),A)/C_0([0,\infty),A). 
\]
In Theorem~\ref{T.Ainfty} we prove the analog of Theorem~\ref{T.C} for $A_{\text\faForward}$. 
The following is  Corollary~\ref{C1}.

\begin{theoremi}\label{T.E} Suppose that $D$ is a separable, unital,   and strongly self-absorbing  \cstar-algebra. Then  all unital $^*$-homomorphisms of $D$ into~$A_{\text\faForward}$ are unitarily equivalent and all unital $^*$-homomorphisms of $D$ into~$\cQ(A\otimes \cK)$ are unitarily equivalent.
\end{theoremi}

There is an obvious diagonal embedding of $A$ into $A_{\text\faForward}$, and we identify~$A$ with its image.  Together with standard methods, Theorem~\ref{T.E} and Theorem~\ref{T.Ainfty} imply the following analog of a vital property of (relative commutants in) ultrapowers and asymptotic sequence algebras (\cite[Theorem~7.2.2]{Ror:Classification}),  appearing here as  Corollary~\ref{T.commutant}.

\begin{corollaryi}
 \label{T.RC} Suppose that $D$ is a separable, unital,   and strongly self-absorbing  \cstar-algebra. If $A$ is a separable and unital \cstar-algebra then the following are equivalent. 
\begin{enumerate}
	\item $A$ is tensorially $D$-absorbing. 
	\item $D$ unitally embeds into $A_{\text\faForward}\cap A'$ . 
\end{enumerate}
\end{corollaryi}

It is not known whether the analogous statement holds for $\cQ(A\otimes \cK)$ (see Conjecture~\ref{Conj.Q}). 
 
\subsection*{Acknowledgments} I would like to thank Ping Wong Ng for communicating Zhang's question to me and to Bruce Blackadar, Jamie Gabe, Ilan Hirshberg,   Narutaka Ozawa, G\'abor Szab\'o, and Mikael R\o rdam for helpful remarks on the earlier versions of this paper. 
Thanks to G\'abor Szab\'o for pointing out to an error in proofs of \S\ref{S.Ainfty} and \S\ref{S.Q(A)} that was as embarrassing as it was inconsequential and to Jamie Gabe for improving the conclusion of Theorem~\ref{T.E} by replacing approximate unitary equivalence with unitary equivalence. Last, but not least, I am indebted to the anonymous referee for an extremely useful referee report.

\section{Terminology, notation, and background on property (T)}\label{S.Terminology}
\subsection{Terminology and  notation}We largely follow \cite{Black:Operator}. Consider the Calkin algebra $\cQ(H):=\cB(H)/\cK(H)$. For a non-unital \cstar-algebra $A$ consider the multiplier algebra $\cM(A)$ and the  corona $\cQ(A):=\cM(A)/A$ (also called outer multiplier algebra). Thus $\cQ(H)$ is the same as $\cQ(\cK)$, where~$\cK$ is the algebra of compact operators on a separable, infinite-dimensional Hilbert space, alternatively denoted $\cK(H)$, as needed. 
 In all of these situations we  let $\pi\colon \cM(A)\to \cQ(A)$ denote the quotient map. 
 
 By 
$
F\Subset X
$
we denote the fact that $F$ is a finite subset of a set $X$. 

If $A$ is a \cstar-algebra or a von Neumann algebra, by $Z(A)$ we denote its center.

\subsection{Strongly self-absorbing \cstar-algebras} A unital \cstar-algebra $D$ is said to be \emph{strongly self-absorbing} if $D$ is isomorphic to $D\otimes D$ via an  isomorphism approximately unitarily equivalent to the $^*$-homomorphism $d\mapsto d\otimes 1_D$ and~$D$ is not isomorphic to $\bbC$. All strongly self-absorbing \cstar-algebras are nuclear and simple, and they satisfy $D\cong D^{\bigotimes\bbN}$  (see \cite{ToWi:Strongly}). Known strongly self-absorbing \cstar-algebras are the  Jiang--Su algebra $\cZ$, UHF algebras of infinite type, $\cO_\infty$, tensor products of UHF algebras of infinite type with~$\cO_\infty$, and $\cO_2$. If all nuclear \cstar-algebras satisfy the Universal Coefficient Theorem, UCT (see e.g., \cite{blackadar1998k}) then there are no other separable strongly self-absorbing \cstar-algebras (this is Corollary 6.7 of \cite{tikuisis2017quasidiagonality}). There exist exact \cstar-algebras without the UCT (\cite{skandalis1988notion}) but the question whether there are nuclear examples, or even strongly self-absorbing ones,   is wide open. 

\subsection{Kazhdan's property (T)} This subsection and \S\ref{S.Schur} contain more background than the experts may want to see. 
A discrete group $\Gamma$ has \emph{Kazhdan's property (T)} if there are $F\Subset \Gamma$ and $\varepsilon>0$ such that for every unitary representation $\rho$ of $\Gamma$ on a Hilbert space~$H$, if there is a vector $\xi$ in~$H$ such that\footnote{When dealing only with a Hilbert space~$H$, its norm will be denoted by $\|\cdot\|$. Later on, when we consider the Hilbert space of Hilbert--Schmidt operators, we will use $\|\cdot\|$ for the operator norm on this space and (variously embellished variations on)  $\|\cdot\|_{\textrm{HS}}$ for the Hilbert--Schmidt norm.}  
\begin{equation} \label{Eq.Inv}
\max_{g\in F} \|\rho(g)\xi-\xi\|<\varepsilon\|\xi\|.
\end{equation}
 for all $g\in F$, (such $\xi\in H$  is called \emph{$F,\varepsilon$-invariant}), then there is a unit vector in $H$ that is invariant for $\rho$. The pair $F,\varepsilon$ is called a \emph{Kazhdan pair} for $\Gamma$. 
Suppose $\rho$ is an irreducible representation of $\Gamma$. Then $\rho$ extends to a representation of the full group \cstar-algebra $\cst(\Gamma)$ (see \cite{BrOz:C*} for information on group \cstar-algebras). Let $p_\rho$ denote the central cover of $\rho$ (\cite[Definition~1.4.2]{BrOz:C*}) in the second dual, $\cst(\Gamma)^{**}$. If $\Gamma$ has Kazhdan's property (T) then~$p_\rho$ belongs to $\cst(\Gamma)$ 
 (\cite[Definition 17.2.3 and Theorem 17.2.4]{BrOz:C*}; as pointed out in \cite{BrOz:C*} in the paragraph preceding Definition 17.2.3, this is a consequence of Schur’s Lemma for property (T) groups, a relative to Lemma~\ref{L.Schur}  below). In this case~$p_\rho$ is called the \emph{Kazhdan projection} of $\rho$.  Kazhdan projections associated to inequivalent irreducible representations are orthogonal. See \cite{bekka2008kazhdan}, \cite{BrOz:C*} for more on property~(T) groups.  For more information on Kazhdan projections see \cite[\S 17]{BrOz:C*}, also \cite[3.7.6]{higson2000analytic}.

\subsection{Schur's Lemma for property (T) groups} 
\label{S.Schur} 
This subsection contains a rehashing of (\cite[Lemma~3.7.8]{higson2000analytic}), Lemma~\ref{L.Schur}, which is  included for convenience of readers not familiar with property (T) groups. 

Suppose that $\rho$ is an action of a discrete group $\Gamma$ on a Hilbert space $H$. 
 For us it will be difficult to keep track on the norm of $\xi$, and we therefore say that $\xi\in H$ is \emph{scaled $F,\varepsilon$-invariant}\footnote{Arguably, a better terminology would be non-scaled $F,\varepsilon$-invariant but this sounds rather excessive and after all the scaling can go either way.}
  if \begin{equation} \label{Eq.Scaled} 
 \max_{g\in F}\|\rho(g)\xi-\xi\|<\varepsilon.
 \end{equation} 
Clearly every $\xi$ is scaled $F,\varepsilon$ invariant for any $F$ and $\varepsilon>2\|\xi\|$,  but in this case the estimate given in Lemma~\ref{L.bekka} below is vacuous. 

 Let $q^\rho$ denote the orthogonal projection to the space of invariant vectors, 
\[
H^\rho:=\{\eta\in H\mid \rho(g)\eta=\eta\text{ for all } g\in \Gamma\}. 
\]
The following is \cite[Proposition 12.1.6]{BrOz:C*} (with $\Gamma=\Lambda$) or  \cite[Proposition~1.1.9]{bekka2008kazhdan} (the case when $F$ is compact, and modulo rescaling). 

\begin{lemma} \label{L.bekka} 
Suppose  $\Gamma$ is a group with Kazhdan's property (T) and $F,\varepsilon$ is a Kazhdan pair for $\Gamma$. If $\Gamma$ acts on a Hilbert space $H$ then every scaled~$F,\delta\varepsilon$-invariant vector $\xi$ satisfies $\|\xi-q^\rho \xi\|<\delta$.  \qed \end{lemma}

Suppose that $K$ and $H$ are Hilbert spaces. 
If $H$ is finite-dimensional then let $\tau_H$ denote the normalized tracial state on $\cB(H)$ and consider the normalized Hilbert—Schmidt norm on $\cB(K,H)$, 
\[
\hHS T:=\tau_H(TT^*)^{1/2}.  
\]
This  differs from the non-normalized Hilbert--Schmidt norm usually used in this context (\cite{higson2000analytic}, \cite{BrOz:C*}) by the scaling factor $\dim(H)^{-1}$.
 
Similarly, if $K$ is finite-dimensional let $\tau_K$ denote the normalized tracial state on $\cB(K)$ and 
consider the normalized Hilbert—Schmidt norm on $\cB(K,H)$, 
\[
\kHS T:=\tau_K(TT^*)^{1/2}.  
\]
Note also that we allow for the possibility that only one of  $H$ and $K$ is finite-dimensional. In this case the normalized Hilbert--Schmidt norm associated with the infinite-dimensional Hilbert space is of course not defined. We will encounter this situation later on.

It is a good moment to get a few easy facts out of the way. If $T\in \cB(K,H)$ then  $\max(\kHS T,\hHS T)\leq \|T\|$ (the right-hand side being the operator norm). Also, if $\dim K\leq \dim  H$, then $\kHS T\geq \hHS T$, more precisely 
\[
\kHS T=\frac{\dim(H)}{\dim(K)}\hHS T.
\]  
We will contract nested subscripts and write   $\nHS \cdot $  for both  $\|\cdot\|_{\textrm{HS},H_n}$ and $\|\cdot\|_{\textrm{HS},K_n}$, similarly $\tau_n$ for both $\tau_{H_n}$ and $\tau_{K_n}$ (spaces $H_n$ and~$K_n$ will never appear in the same context). 

Lemma~\ref{L.Schur} below is a minor variation on Schur's Lemma for property (T) groups (\cite[Lemma~3.7.8]{higson2000analytic}).

\begin{lemma} \label{L.Schur}  Suppose  $\Gamma$ is a property (T) group with Kazhdan pair $F,\varepsilon$ and   $\rho_j$ is a  representation of $\Gamma$ on $K_j$ for $j=1,2$ such that $K_1$ is finite-dimensional and  some $T \in \cB(K_1,K_2)$ satisfies 
	\[
	\max_{g\in F} \|T\rho_1(g)-\rho_2(g)T\|<\varepsilon\delta.
	\] 
\begin{enumerate}
	\item \label{1.Schur} If $K_1=K_2$ and $\rho_1=\rho_2$ is irreducible, then  $\oneHS{T-\tau_1(T)}<\delta$. 
	\item \label{2.Schur} If $\rho_1$, $\rho_2$ have no isomorphic subrepresentations, then $\oneHS T<\delta$. 
\end{enumerate}
\end{lemma}

\begin{proof} Since $\oneHS\cdot\leq \|\cdot\|$, we have
\begin{equation}\label{Eq.HS-inequality}
\max_{g\in F} \oneHS{T\rho_1(g)-\rho_2(g)T}<\varepsilon\delta
\end{equation}
and we will show that this inequality implies both conclusions.  
Define an action  $\sigma=\sigma_{\rho_1,\rho_2}$ of $\Gamma$ on  $\cB(K_1,K_2)$ by letting, for $T\in \cB(K_1,K_2)$, 
\[
\sigma(g)(T):=\rho_2(g) T \rho_1 (g)^*. 
\]
A calculation shows that  for all $g$ and $h$ in $\Gamma$ we have $\sigma(g)\in \cB(K_1,K_2)$ and   $\sigma(gh)=\sigma(g)\circ \sigma(h)$. Thus $\sigma$ is  a representation  of $\Gamma$ on $\cB(K_1,K_2)$. Consider the space $\cB(K_1,K_2)^\sigma$ of   vectors invariant for $\sigma$, and let $Q$ be the projection onto this space. 

Since $F,\varepsilon$ is a Kazhdan pair for  $\Gamma$,  by \eqref{Eq.HS-inequality} and  Lemma~\ref{L.bekka} we have 
\begin{equation}\label{Eq.Q}
\oneHS {T-Q(T)}<\delta. 
\end{equation}
Now consider the two specific cases from the statement of the lemma. 

\eqref{1.Schur} If $\rho_1=\rho_2$ and it is irreducible, then (identifying scalars with  scalar matrices), Schur's Lemma (\cite[Lemma~3.7.7]{higson2000analytic}) implies $\cB(K_1,K_1)^\sigma=\bbC$ and $Q(T)=\tau_1(T)$. Therefore \eqref{Eq.Q} reduces to  $\oneHS{T-\tau_1(T)}<\delta$.

\eqref{2.Schur} Suppose for a moment that $\rho_1$ is irreducible and not isomorphic to a subrepresentation of $\rho_2$.  Then Schur's Lemma (\cite[Lemma~3.7.7]{higson2000analytic}) implies that  there are no nontrivial intertwiners for $\rho_1$ and $\rho_2$, in other words   $\cB(K_1,K_2)^\sigma=\{0\}$. Hence $Q(T)=0$ and \eqref{Eq.Q} reduces to  $\oneHS T<\delta$. 

Suppose $\rho_1$ is not  irreducible.  Since $K_1$ is finite-dimensional,~$\rho_1$ is a direct sum of irreducible representations. By applying the previous paragraph to each one of these irreducible representations, we have $\cB(K_1,K_2)^\sigma=\{0\}$ and  $Q(T)=0$ and therefore $\oneHS T<\delta$.    \end{proof}

\section{Wassermann's \cstar-algebra $\cW(\Gamma)$}\label{S.Gamma.A}
Suppose $\Gamma$ is a property (T)  group with infinitely many inequivalent irreducible representations on finite-di\-men\-sio\-nal Hilbert spaces (e.g., a residually finite property (T) group such as $\SL_3(\bbZ)$).  Fix a Kazhdan pair,  $F\Subset \Gamma$ and $\varepsilon>0$. Let $\rho_n$, for $n\in \bbN$, be an enumeration of irreducible representations of $\Gamma$ on finite-dimensional Hilbert spaces. Since there are only finitely many inequivalent representations of $\Gamma$ on every fixed finite-dimensional Hilbert space $H$ (\cite[Corollary~3]{wassermann1991c}), we may choose the enumeration  so that $\dim(H_n)\leq \dim(H_{n+1})$ for all $n$ and  $\dim(H_n)\to \infty$ as $n\to \infty$. For each $n$ let $p_n$ be the Kazhdan projection associated with $\rho_n$ in $\cst(\Gamma)$. 
 Each representation $\rho_n$ of $\Gamma$ on $H_n$ uniquely extends to a representation of the full group algebra  $\cst(\Gamma)$ on the same space. We will slightly abuse the notation and denote the latter by~$\rho_n$.  
Let $\rho=\bigoplus_n \rho_n$, and $H:=\bigoplus_n H_n$. Then $H$ is an infinite-dimensional, separable Hilbert space and $\rho\colon \cst(\Gamma)\to \cB(H)$ is a unital $^*$-homomorphism such that $\rho^{-1}(\cK(H))$ includes  all $p_n$. 

\begin{definition}\label{D.W(Gamma)} If $\Gamma$  is a property (T) group with infinitely many irreducible representations $\rho_n\colon \Gamma\to \cB(H_n)$, for $n\in \bbN$,  on finite-dimensional Hilbert spaces and $\rho=\bigoplus_n \rho_n$ is considered as a representation of the full group algebra $\cst(\Gamma)$ as in  the previous paragraph,  then with 
$\pi\colon \cB(H)\to \cQ(H)$ denoting the quotient map let 
\begin{equation*}\label{Eq.W(Gamma)}
	\cW(\Gamma):=\pi\circ \rho[\cst(\Gamma)]. 
\end{equation*}
With $p_n$ denoting the Kazhdan projection associated with $\rho_n$, let
\begin{equation*}
	\calD:=\{T\in \cB(H)\mid p_m T p_n \neq 0\text{ implies } m=n\}. 
\end{equation*}
\end{definition}

In other words, $\calD=\prod_n p_n \cB(H)p_n=\prod_n \cB(H_n)$.  

Had we chosen an orthonormal basis of $H$ whose subsets span the $H_n$'s, then~$\calD$ would be a special case of a von Neumann algebra of the form~$\calD[{\mathbf E}]$ studied in \cite[\S 9.7.1]{Fa:STCstar}. These von Neumann algebras are of central importance both in author's construction of outer automorphisms of $\cQ(H)$ (\cite[\S 17.1]{Fa:STCstar}---the original construction of outer automorphisms given in \cite{PhWe:Calkin} uses separable \cstar-subalgebras of $\cQ(H)$ instead, and one of the key lemmas used in this proof plays a role in the proof of Corollary~\ref{C1} below) and in the proof that all automorphisms of $\cQ(H)$ are inner (\cite{Fa:All}, \cite[\S 17.2--17.8]{Fa:STCstar}). 

Note that the center $Z(\calD)$ of $\calD$ is equal to 
\begin{equation}\label{Eq.Z(D)}
\textstyle Z(\calD)=\{\sum_n \lambda_n p_n\mid (\lambda_n)\in \ell_\infty\}
\end{equation}
and that $\pi[Z(\calD)]\subseteq \cQ(H)\cap \cW(\Gamma)'$ because each $p_n$ is invariant for $\rho(T)$ for all $T\in \cst(\Gamma)$. 


Since for a scalar $\lambda$ we have $\|\lambda p_n\|=|\lambda|=\nHS{\lambda p_n}$,\footnote{A clarification for the readers who skipped \S\ref{S.Schur}: this is the normalized Hilbert--Schmidt norm.} every $T\in Z(\calD)$ satisfies 
\[
\|T\|=\sup_n \nHS{Tp_n}. 
\]
The following convention, going back at least to \cite{We:Set}, will considerably simplify the notation, and it will not lead to confusion if the reader keeps it in mind.

\begin{convention} \label{Conv}
In Proposition~\ref{P.linfty} and elsewhere we denote the elements of the Calkin algebra by dotted letters, with understanding that if $\dot a\in \cQ(H)$ then $a\in \cB(H)$ is a lift of $\dot a$. 
\end{convention}

\begin{proposition}\label{P.linfty} Suppose that $\Gamma$,  $\cW(\Gamma)$, and $\calD$ are as in Definition~\ref{D.W(Gamma)}. 
Then, with $A:=\pi[Z(\calD)]$,  there is  a $^*$-ho\-mo\-mor\-phism 
\[
\Psi\colon \cQ(H)\cap \cW(\Gamma)'\to A
\]	
whose restriction to $A$ is the identity map.  Thus, if $I:=\ker(\Psi)$ then the short exact sequence
\[
0\to I \to \cQ(H)\cap \cW(\Gamma)' \toPsi A\to 0
\]
splits. 
\end{proposition}
Prior to proving this proposition, we show that Theorem~\ref{T.B} is its immediate consequence.

\begin{proof}[Proof of Theorem~\ref{T.B}] Let $\Gamma$ and $\cW(\Gamma)$ be as in the first paragraph of~\S\ref{S.Gamma.A}. We will prove that $B:=\cW(\Gamma)$ is as required.   Let $C$ be a unital and simple \cstar-algebra, and suppose that $\Phi$ 
	is a unital $^*$-homomorphism from $C$ into $\cQ(H)\cap \cW(\Gamma)'$. If $\Psi$ is as constructed in Proposition~\ref{P.linfty}, then (since $C$ is simple), $\Psi\circ \Phi$ is a unital, and therefore injective, *-homomorphism from~$C$ into the range of~$\Psi$.  However, the range of $\Psi$ is abelian, and therefore $C\cong \bbC$.  
\end{proof}

The proof of Proposition~\ref{P.linfty} commences with a sequence of statements, some of them self-explanatory,  phrased  as a proposition.  

\begin{proposition}\label{P.W(Gamma)}
Suppose that $\Gamma$, $\rho\colon \cst(\Gamma)\to \cB(H)$,  $\cW(\Gamma)$, and $\calD$ are as in Definition~\ref{D.W(Gamma)}. Then 
we have the following. 
\begin{enumerate}
	\item \label{1.P.W(Gamma)} $\rho[\cst(\Gamma)]''=\calD$. 
	 \item \label{2.P.W(Gamma)} $\rho[\cst(\Gamma)]'=\calD’=\wst(p_n: n\in \bbN) =Z(\calD)$. 
	 \item \label{3.P.W(Gamma)} A conditional expectation $E\colon \cB(H)\to \calD'$ is defined by ($\tau_n$ is the normalized trace on $\cB(H_n)$)
	 \begin{equation}\label{eq.E}
	 E(T)=\sum_n \tau_n(p_n T p_n) p_n. 
	 \end{equation}
	\item \label{4.P.W(Gamma)} $\pi[\calD']=\cQ(H)\cap \pi[\calD]'\subseteq \cQ(H)\cap \cW(\Gamma)'$. 
	\item \label{5.P.W(Gamma)} There is a conditional expectation $\dot E\colon \cQ(H)\to \cQ(H)\cap \pi[\calD]'$ such that $\dot E(\dot T)=\pi(E(T))$ for all $T\in \cB(H)$. 
\item  \label{6.P.W(Gamma)} Every $T$ such that $\dot T\in \cQ(H)\cap \cW(\Gamma)'$	 satisfies
	\begin{align*}
	\label{Eq.Tpn} \lim_n \nHS{T p_n -\tau_n(p_n T p_n)p_n}&=0, \\
\lim_n \nHS{T p_n -E(T)p_n}&=0, \text{ and}\\
\lim_n \nHS{p_n T -E(T)p_n}&=0.
\end{align*}
\end{enumerate}
\end{proposition}

\begin{proof} 
\eqref{1.P.W(Gamma)} For every $n$, $\rho_n$ is irreducible and $H_n$ is finite-dimensional, and therefore $p_n \rho[\cst(\Gamma)]p_n=\rho_n[\cst(\Gamma)] =\cB(H_n)$ for all $n$.  Since the $\rho_n$'s are pairwise inequivalent, there are no nontrivial intertwiners  and by a consequence of the Glimm--Kadison transitivity theorem (\cite[Proposition~3.5.2]{Fa:STCstar}) we have   
\[
\rho[\cst(\Gamma)]''=\prod_n \cB(H_n)=\calD.
\] 

\eqref{2.P.W(Gamma)} By \eqref{1.P.W(Gamma)},  $\rho[\cst(\Gamma)]'=\calD'=\wst(p_n: n\in \bbN)$ and $Z(\calD)=\calD\cap \calD’=\calD’$. 

\eqref{3.P.W(Gamma)} Follows from \eqref{2.P.W(Gamma)}, since $\calD=\prod_n \cB(H_n)$. 

\eqref{4.P.W(Gamma)}
We first  prove that
$\pi[\calD']=\cQ(H)\cap \pi[\calD]'$ follows from the results in~\cite{JohPar}. 
Let $H$ be a Hilbert space and let $\calD$ be a von Neumann algebra acting on $H$. Property $P_1$ (\cite[p. 39]{JohPar}) asserts that if $b\in \cB(H)$ and $ab-ba\in \cK(H)$ for all $a\in \calD$, then $b\in \calD+\cK(H)$.   We need to verify that~$\calD$ has property~$P_1$. 

 For each $n$, $p_n$ belongs to the center of $\calD$ and $\calD p_n=\cB(H_n)$ is a finite type I factor. By \cite[Lemma 3.1]{JohPar}, every derivation of $\calD p_n$ into $\cK(H)$ is inner, and (in the terminology of \cite[p. 39]{JohPar}), $\calD p_n$ has property $P_2$.  \cite[Theorem~3.4]{JohPar} now implies that every derivation of $\calD$ is inner, and by \cite[Lemma~1.4]{JohPar} $\calD$  has property $P_1$ as well, as required.

We proceed to prove $\pi[\calD']\subseteq \cQ(H)\cap \cW(\Gamma)'$. Since $\calD=\rho[\cst(\Gamma)]''$, we have $\pi[\calD]\supseteq \pi[\cst(\Gamma)]=\cW(\Gamma)$. Therefore  
$
\cQ(H)\cap \cW(\Gamma)'\supseteq \cQ(H)\cap \pi[\calD]'$.

\eqref{5.P.W(Gamma)} Since $E$ as defined in \eqref{eq.E} sends compact operators to compact operators, it defines a conditional expectation $\dot E\colon \cQ(H)\to \cQ(H)\cap \pi[\calD]'$ by $\dot E(\dot T)=\pi(E(T))$, as required.

\eqref{6.P.W(Gamma)}
		For $Y\subseteq \bbN$ it will be convenient to write 
	\[
	H_Y:=\bigoplus_{m\in Y} H_m
	\]
	and let $\rho_Y:=\bigoplus_{m\in Y} \rho_m$, so that $\rho_Y$ is a   representation of $\cst(\Gamma)$ on $H_Y$. 
	 Finally let (the right-hand side is strongly convergent) 
	 \[
	 p_Y:=\sum_{m\in Y} p_m.
	 \] 
	Fix $T\in \cB(H)$ such that $\dot T\in \cQ(H)\cap \cW(\Gamma)'$. Then $\rho(g)T-T\rho(g)$ is compact for all $g\in \Gamma$. We can therefore fix $0<\delta<1/5$ and~$n_0$ such that for all $g\in F$ we have 
	\begin{equation*}
		\textstyle \| (\rho(g)T-T\rho(g))(1-\sum_{m<n_0} p_m)\|<\delta\varepsilon 
	\end{equation*}
and this implies
\begin{equation}\label{Eq.T}
\|p_X (\rho(g) T\rho(g)^* - T)p_Y\|<\delta\varepsilon
\end{equation}
for all $g\in F$ and all  $X$ and  $Y$ included in $[n_0,\infty)$. 
	Write $T\sim_\varepsilon S$ for $\|T-S\|<\varepsilon$ (this is the operator norm).  By  \eqref{Eq.T},  every $Y\subseteq [n_0,\infty)$ and every $n\geq n_0$ satisfy
	\begin{equation}\label{Eq.rhoY}
		\rho_Y(g) p_Y T p_n \rho_n (g)^*=p_Y \rho(g)T \rho(g)^* p_n\sim_{\delta\varepsilon} p_Y T p_n. 
	\end{equation}
	Since all  $\rho_j$ are inequivalent and irreducible,  if $n\notin Y$ then $\rho_n$ and $\rho_Y$ have no isomorphic subrepresentations.   By \eqref{Eq.rhoY} and  Lemma~\ref{L.Schur} \eqref{2.Schur}, 
	\begin{equation}\label{Eq.pYTpn}
		\nHS{p_Y T p_n}<\delta. 
	\end{equation}
	On the other hand, since $p_{[0,n_0)} T$ has finite rank, there is $n_1>n_0$ such that  all $n\geq n_1$  satisfy
	\begin{equation}\label{Eq.p0n0Tpn}
		\nHS{p_{[0,n_0)} T p_n}<\delta. 
	\end{equation}
	By \eqref{Eq.p0n0Tpn} and \eqref{Eq.pYTpn} applied to $Y:=[n_1,\infty)\setminus\{n\}$, every $n\geq n_1$ satisfies $\nHS{p_{\bbN\setminus \{n\}} T p_n}<2\delta$. Equivalently, 
	\begin{equation}\label{Eq.pnpn}
		\nHS {Tp_n- p_n T p_n}<2\delta. 
	\end{equation}
	On the other hand, since $\rho_n$ is irreducible, Lemma~\ref{L.Schur} \eqref{1.Schur} implies  
	\[
	\nHS{ p_n T p_n -\tau_n(p_n T p_n) p_n}<\delta
	\]
	which together with \eqref{Eq.pnpn} implies 
	\begin{equation}\label{Eq.Tpndelta}
		\nHS{T p_n -\tau_n(p_n T p_n)p_n}<3\delta. 
	\end{equation}
Since $\delta>0$ was arbitrary and \eqref{Eq.Tpndelta} holds for every large enough $n$, this completes the proof of 
	$
	\label{Eq.Tpn} \lim_n \nHS{T p_n -\tau_n(p_n T p_n)p_n}=0$, which is 
the first equation in \eqref{6.P.W(Gamma)}.  The second equation,
	$\lim_n \nHS{T p_n -E(T)p_n}=0$,
 follows since $E(T)=\sum_n \tau_n (p_n T p_n)p_n$. The third equation is obtained by replacing~$T$ with $T^*$ in the second.  
\end{proof}

\begin{proof}[Proof of Proposition~\ref{P.linfty}] With $\Gamma$, $\cW(\Gamma)$ as in Definition~\ref{D.W(Gamma)}, 
 let $\Psi$ be the restriction of $\dot E$ as in Proposition~\ref{P.W(Gamma)} \eqref{5.P.W(Gamma)} to $\cQ(H)\cap \cW(\Gamma)'$ and let 
 \[
 A:=\pi[Z(\calD)].
 \] 
 Since \eqref{2.P.W(Gamma)}  of Proposition~\ref{P.W(Gamma)} implies that $Z(\calD)=\calD'=\wst(p_n: n\in \bbN)$, it is isomorphic to $\ell_\infty(\bbN)$, hence 
 $A$ is isomorphic to $\ell_\infty(\bbN)/c_0(\bbN)$.  In particular,~$A$ is abelian, and, by \eqref{4.P.W(Gamma)} of Proposition~\ref{P.W(Gamma)},  included in $\cQ(H)\cap \cW(\Gamma)'$. 

	Fix $T\in \cB(H)$ such that $\dot T$ is in the domain of $\Psi$.	Being a restriction of~$\dot E$, $\Psi$ satisfies 
	\[
	\textstyle\Psi(\dot T):=\pi(\sum_n \tau_n(p_n T p_n) p_n). 
	\]
	It remains to prove that $\Psi$ is a $^*$-homomorphism. Since it is a conditional expectation, it suffices to check that it is multiplicative. This will be proven by multiple uses of Proposition~\ref{P.W(Gamma)} \eqref{6.P.W(Gamma)}.

	If $\dot S$ and $\dot T$ belong to $\cQ(H)\cap \cW(\Gamma)'$, then for every $n$ and scalars $\eta$ and~$\lambda$ we have (using $\nHS{ab}\leq \|a\|\nHS b$) 
	\begin{align*}
	\nHS{STp_n -\eta\lambda p_n}&\leq 
	\nHS{ST p_n - S\lambda p_n}+
	\nHS{S \lambda  p_n - \eta\lambda p_n}\\
	&\leq \|S\|\nHS{T p_n -\lambda p_n}
	+|\lambda|\nHS{ S p_n -\eta p_n}. 
	\end{align*}
	Thus if $E(T)=\sum_n \lambda_n p_n$ and $E(S)=\sum_n \eta_n p_n $, then the previous formula together with  the first equation of Proposition~\ref{P.W(Gamma)} \eqref{6.P.W(Gamma)} applied to $S$ and $T$ implies
	\(
	\nHS{ST p_n -\eta_n \lambda_n p_n}\to 0
	\),
and the same equation applied to $ST$ implies 
	 	$\lim_n \nHS{ST p_n -\tau_n(p_n ST p_n)p_n}=0$.  Therefore
	 	\[
	 	\lim_n \tau_n(p_n ST p_n )-\tau_n(p_n S p_n)\tau_n(p_n T p_n)=0 .
	 	\] 
	 After passing to the quotient, this  gives $\dot E(\dot S\dot T)=\pi(\sum_n \eta_n \lambda_n p_n)=\dot E(\dot S)\dot E(\dot T)$, as required.

Thus $\Psi$ is multiplicative, hence it is a $^*$-homomorphism whose restriction to its range $A$ is the identity map.  Thus, if $I:=\ker(\Psi)$ then  the short exact sequence
\[
0\to I \to \cQ(H)\cap \cW(\Gamma)'\toPsi A\to 0
\]
splits with $A\ni a\mapsto a+I$ being the right inverse of $\Psi$. 
\end{proof}

Although it is not used elsewhere in the present paper, I find it that the following proposition provides sufficient insight into the structure of the relative commutant of $\cW(\Gamma)$ in $\cQ(H)$ and the structure of $\Ext(\cW(\Gamma))$ to merit its inclusion. 

\begin{proposition}
\label{L.01law} Suppose that $\Gamma$,  $\cW(\Gamma)$, and $p_n$, for $n\in \bbN$,  are as in Definition~\ref{D.W(Gamma)}.  
\begin{enumerate}
\item\label{1.01law} 	For every projection $p\in \cB(H)$ such that $\dot p\in \cQ(H)\cap \cW(\Gamma)'$ there is $X=X_p\subseteq \bbN$ which satisfies  (considering $pp_n$ as an element of $\cB(H_n,H)$)
\[
\lim_{n\in X,\\ n\to \infty} \nHS{ pp_n}=1
\text{ 
and }
\lim_{n\notin X,\\ n\to \infty} \nHS{ pp_n}=0. 
\]
\item\label{2.01law} If $\dot p$ and $\dot q$ are Murray--von Neumann equivalent projections belonging to  $\cQ(H)\cap \cW(\Gamma)'$, then the symmetric difference $X_p\Delta X_q$ is finite. 
\end{enumerate}
\end{proposition}

\begin{proof} With the $^*$-homomorphism $\Psi\colon \cQ(H)\cap \cW(\Gamma)'\to \pi[Z(\calD)]$  as in Proposition~\ref{P.linfty}, $\Psi(\dot p)$ is a projection in $\pi[Z(\calD)]$. Since $Z(\calD)$, being a von Neumann algebra,  has real rank zero, by \cite[Lemma~3.1.13]{Fa:STCstar} $\Psi(\dot p)$ can be lifted to a projection $q$ in $Z(\calD)$. This projection is of the form $q=\sum_{n\in X_p} p_n$ for some $X_p\subseteq \bbN$. 
	Writing $\tau_n$ for the normalized trace on $\cB(H_n)$ as in the proof of Proposition~\ref{P.linfty},  we have  $\dot E(\dot p)=\dot E(\dot q)=\pi(\sum_n \tau_n(p_n p p_n)p_n)$. Since $E(p)-\sum_{n\in X_p} p_n$ is compact, we have\footnote{If one of $X_p$ or $\bbN\setminus X_p$ is finite, then only one of the two displayed limits is well-defined and relevant, and we ignore the other one.} 
\[
\lim_{n\to \infty, n\in X_p} \tau_n(p_n p p_n)=1
 \text{ and }\lim_{n\to \infty, n\notin X_p} \tau_n(p_n p p_n)=0.
 \]
  Since $\nHS{p p_n}^2=\tau_n( p_n p p_n)$
  the set $X_p$ is as required.

\eqref{2.01law} Suppose that $\dot p$ and $\dot q$ are Murray--von Neumann equivalent projections in $\cQ(H)\cap \cW(\Gamma)'$ and let $\dot v$ be a partial isometry in $\cQ(H)\cap \cW(\Gamma)'$ such that $\dot v^* \dot v=\dot v\dot v^*$. 

Since the range of $\Psi$ as provided by Proposition~\ref{P.linfty} is abelian, we have  $\Psi(\dot p)=\Psi(\dot v^* \dot v)=\Psi(\dot v \dot v^*)=\Psi(\dot q)$ and therefore $X_p$ and $X_q$ correspond to the same projection in $\pi[\calD]$. This is equivalent to $X_p\Delta X_q$ being finite. 
\end{proof}

 Perhaps it should be noted that the converse of Proposition~\ref{L.01law} \eqref{2.01law} is not true. For example, if $\dot p$ corresponds to the trivial extension of $\cW(\Gamma)$ then~$X_p$ is finite, but $\dot p$ is not equal to $0$.

\section{$D$-saturation of the path algebra $A_{\text\faForward}$}\label{S.Ainfty}

The main result of this section, Theorem~\ref{T.Ainfty}, is a warmup for Theorem~\ref{T.Q(A)} that will be used in the proof of Theorem~\ref{T.A}. Theorem~\ref{T.Ainfty} will not be used in this proof and readers not interested in path algebras (as well as those not requiring a warmup) can proceed directly to \S\ref{S.Q(A)}. 

Suppose that $D$ is a separable and unital \cstar-algebra. Following \cite{FaHaRoTi:Relative}, we say that  a (nonseparable) unital \cstar-algebra $C$ is \emph{$D$-saturated} if for every separable \cstar-subalgebra of $C$ there is an injective, unital $^*$-homomorphism from $D$ into $C\cap A'$. Prime examples of $D$-saturated algebras are the ultrapower $A_\cU$ and an asymptotic sequence algebra $\ell_\infty(A)/c_0(A)$ of a tensorially $D$-absorbing \cstar-algebra $A$ (see \cite{effros1978c}, \cite[Theorem~7.2.2]{Ror:Classification},  and \cite{FaHaRoTi:Relative}).

For a \cstar-algebra $A$ consider the path algebra 
\[
A_{\text\faForward}:=C_b([0,\infty), A)/C_0([0,\infty),A). 
\]

\begin{theorem}\label{T.Ainfty}
Suppose that $D$ is a separable, unital,   and strongly self-absorbing  \cstar-algebra. If $A$ is separable, unital, and tensorially $D$-absorbing then  $A_{\text\faForward}$ is $D$-saturated. 
\end{theorem}

\begin{proof} Since $D\cong D^{\bigotimes \bbN}$, we can identify $A$ with $A\otimes D^{\otimes\bbN}$. Letting $D_j$ denote the $j$-th copy\footnote{We adopt the convention that $0\in \bbN$.} of $D$ in $D^{\bigotimes \bbN}$,  we have that  $A$ is the inductive limit of \cstar-algebras $A_0\cong A$ and $A_n:=A\otimes \bigotimes_{j<n} D_j$, for $n\geq 1$. Let $\Theta_n\colon D\to A$ be the embedding that identifies $D$ with $D_n$, so that 
\[
\Theta_n(d):=1_A\otimes \bigotimes_{j<n} 1_{D_j}\otimes d \otimes \bigotimes_{j>n} 1_{D_j}. 
\] 
By \cite[Remark~3.3]{winter2011strongly} $D$ is $K_1$-injective hence   \cite[Theorem~2.2]{dadarlat2009kk} implies that  the flip automorphism of $D\otimes D$ (the one that interchanges $a\otimes b$ with $b\otimes a$) is \emph{strongly asymptotically equivalent to the identity}. This means that there is a continuous path of unitaries $u_t$, for $0\leq t<1$, in $D\otimes D$ such that $u_0=1$ and $\lim_{t\to 1} u_t (a\otimes b) u_t^*=b\otimes a$. 
Fix $m<n$ and identify $D_m\otimes D_n$ with $D\otimes D$. Then the previous sentence gives a continuous path of unitaries $w^{m,n}_t$, for $0\leq t< 1$, in $A$ such that $w^{m,n}_0=1$ and for all $d\in D$ we have 
\[
\lim_{t\to 1}  w^{m,n}_t \Theta_m(d) (w^{m,n}_t)^*=\Theta_n(d). 
 \]
 In addition, $w^{m,n}_t\in A\cap A_m'$ for all $t$. 
If $m=n$ then for definiteness let $w^{m,n}_t=1$ for all $t$.  
 
 We are now ready to prove that $A_{\text\faForward}$ is $D$-saturated. Fix a separable \cstar-subalgebra $B$ of $A_{\text\faForward}$, and let $\{a_n\mid n\in \bbN\}$, be a subset of $C_b([0,\infty), A)$ such that (with $\pi$ denoting the quotient map) $\pi(a_n)$, for $n\in \bbN$, is dense in $B$. For each $m$ and $\varepsilon>0$ define $f_{m,\varepsilon}\colon [0,\infty)\to \bbN$ by 
 \[
 f_{m,\varepsilon}(t)=\min\{j\mid \dist(a_m(t), A_j)<\varepsilon\}. 
 \]
 This function is clearly upper semicontinuous. 
 
 If $d\in D$, $\|d\|\leq 1$,  $n>f_{m,\varepsilon}(t)$, and $b\in A_{n-1}$ is such that $\|b-a_m(t)\|<\varepsilon$, we then have 
 \begin{equation}\label{eq.estimate}
 	\|[\Theta_n(d),a_m(t)]\|=\|[\Theta_n(d),a_m(t)-b]\|\leq 2\|d\|\|a_m(t)-b\|<2\varepsilon.
 \end{equation}
   For all $n\in \bbN$, the value 
 \[
 g(n):=\max\{f_{m,1/(n+1)}(t)\mid m\leq n+1\text{ and }  t\leq n+1\} 
 \]
 is finite by upper semicontinuity of $f_{m,\varepsilon}$ and compactness of $[0,n+1]$.  
Clearly~$g$ is nondecreasing. 

Define $\Theta\colon D\to C_b([0,\infty), A)$ by first letting, for $n\in \bbN$,  
\[
\Theta(d)(n):=\Theta_{g(n)}(d). 
\]
For $n\in \bbN$ and $0\leq t<1$ let 
\[
\Theta(d)(n+t):=w^{g(n),g(n+1)}_t \Theta_{g(n)}(d) (w^{g(n),g(n+1)}_t)^*. 
\]
By the choice of the unitaries $w^{g(n),g(n+1)}_t$ we have 
\[
\lim_{t\to 1} 
w^{g(n),g(n+1)}_t \Theta_{g(n)}(d) (w^{g(n),g(n+1)}_t)^*=\Theta_{g(n+1)}(d),
\]
 and therefore $\Theta(d)$ is continuous, $\|\Theta(d)(n+t)\|=\|d\|$ for all $n$ and $t$, hence  $\Theta(d)\in C_b([0,\infty), A)$. 
Also,  \eqref{eq.estimate} and $w_t^{g(n),g(n+1)}\in A\cap A_g(n)'$ for all $n$ and $t$ together imply that  for every $m$ and every $d\in D$ we have $\lim_{t\to \infty}\|[a_m(t), \Theta(d) (t)]\|=0$, and therefore $\pi\circ\Theta[D]\subseteq A_{\text\faForward}\cap B'$. 
It is clear that $\pi\circ \Theta\colon D\to A_{\text\faForward}$ is unital and injective. Since $B$ was an arbitrary separable \cstar-subalgebra of $A_{\text\faForward}$, we conclude that $A_{\text\faForward}$ is $D$-saturated. 
\end{proof}

\section{$D$-saturation of the corona $\cQ(A\otimes \cK)$}
\label{S.Q(A)}

In this section we prove the folloqing analog of Theorem~\ref{T.Ainfty}. 

\begin{theorem}\label{T.Q(A)}
	Suppose that $D$ is a separable, unital,   and strongly self-absorbing  \cstar-algebra. If $A$ is separable, unital, and tensorially $D$-absorbing then  $\cQ(A\otimes \cK)$ is $D$-saturated. 
\end{theorem}

The proof of Theorem~\ref{T.Q(A)} is an elaboration of the proof of Theorem~\ref{T.Ainfty}, and as such it also depends on \cite[Theorem~2.2]{dadarlat2009kk} in an essential way. 
The additional ingredient in this proof is Lemma~\ref{L.diagonalization} below. The origins of this lemma can be traced back at least to Voiculescu’s tri-diagonal argument (\cite{voiculescu1976non}, more recently \cite{Voi:Countable}) and more directly  to \cite[Theorem~3.1]{Ell:Derivations} (see \cite{Ell:Ideal}) and \cite[Lemma~1.2]{Fa:All}. The simplest form of these lemmas shows that every element of the corona of a $\sigma$-unital \cstar-algebra that has an approximate unit consisting of projections is the sum of two block-diagonal elements. The appropriately uniform version of these lemmas for countable sets of operators is indispensable in the study of automorphisms of coronas,  both in constructing outer automorphisms (when combined with the Continuum Hypothesis or its weakenings, see \cite[Lemma~4.3]{CoFa:Automorphisms}, \cite[\S 17.1]{Fa:STCstar}) and in proving that all automorphisms of a given corona are inner, or of a very simple form (when combined with forcing axioms, see \cite[\S 17]{Fa:STCstar}, \cite{mckenney2018forcing}, and  \cite{vignati2018rigidity}).

The original construction of outer automorphisms of the Calkin algebra~(\cite{PhWe:Calkin}) uses a different approach that is more directly related to Voi\-cu\-les\-cu’s tri-diagonal trick. This is the unfolding argument of \cite[Proposition~1.4]{PhWe:Calkin} and \cite[\S 3]{manuilov2004theory}, used in the proof of Corollary~\ref{C1} below. The two approaches appear to be interchangeable for many purposes. 

  In the context of (strongly) self-absorbing \cstar-algebras similar methods had been  used in \cite[Theorem~4.5]{Kirc:Central} and \cite[\S 4]{ToWi:Strongly}, and the variant of these results proven in \cite{dadarlat2009kk} is used in the proof of Lemma~\ref{L.discretization}  below.   As the referee pointed out, Lemma~\ref{L.diagonalization} can also be considered as an instance of the `2-coloured equivalence’ used in Elliott’s classification program. The first instance of this was the seminal 2-coloured equivalence of projections (\cite[Lemma~2.1]{matui2014decomposition}) that  has been perfected further to obtain a method for presenting a *-homomorphism into a massive $\cZ$-stable \cstar-algebra as the sum of two order zero completely positive contractions in \cite[\S 6]{bosa2019covering}. At the hindsight, this is very similar to what Lemma~\ref{L.diagonalization} and its predecessors  (including the tri-diagonal-type arguments) do for *-homomorphisms of separable \cstar-algebras, with ultraproducts in their codomain replaced by coronas of $\sigma$-unital \cstar-algebras, with the difference that $\cZ$-stability is in our situation not needed.

The remainder of this section is devoted to the proof of Theorem~\ref{T.Q(A)}. 
Suppose that $D$ and $A$ are as in Theorem~\ref{T.Q(A)}. As in the first paragraph of the proof of Theorem~\ref{T.Ainfty}, we write $A=\lim_n A_n$, where $A_0\cong A$ and $A_{n+1}=A_n\otimes D_n$, with $D_n$ being  an isomorphic copy of $D$. As before, we identify $D_n$ with a unital \cstar-subalgebra of $A$ and  let $\Theta_n\colon D\to A$ be the embedding that identifies $D$ with $D_n$. 
We will need the following discretization of \cite[Theorem~2.2]{dadarlat2009kk}.

\begin{lemma} \label{L.discretization} 
Suppose that $D$ and $A$ are as in Theorem~\ref{T.Q(A)} and that $A_j$, $D_j$, and $\Theta_j$ are as in the previous paragraph. Then for every $\varepsilon>0$ and every $F\Subset D$ there is $k(F,\varepsilon)\in \bbN$ such that for all $m<n$ there are unitaries $w_j=w^{m,n,\varepsilon}_j(F)$, for $j<k(F,\varepsilon)$, in $A$ which satisfy the following. 
\begin{enumerate}
	\item $w_0=1$. 
	\item $\|w_j-w_{j+1}\|<\varepsilon$ for all $j<k(F,\varepsilon)-1$. 
	\item $w_j\in A\cap A_m'$ for all $j<k(F,\varepsilon)$. 
	\item \label{4.discretization} $\|w_{k(F,\varepsilon)-1}\Theta_m (d)w_{k(F,\varepsilon)-1}^*-\Theta_n(d)\|<\varepsilon$ for all $d\in F$. 
\end{enumerate}
\end{lemma}

Clearly the unitaries $w^{m,n,\varepsilon}_j(F)$ depend on the presentation of $A$  as $\lim_n A_n$, but this presentation is fixed once and for all. It is important that $k(F,\varepsilon)$ does not depend on the choice of $m$ and $n$. 

\begin{proof} We may let $w^{m,m,\varepsilon}_j(F):=1$ for all $F,m=n$ and $j<k(F,\varepsilon)$ (with $k(F,\varepsilon)$ to be determined shortly)  and assume that $m<n$.  By \cite[Remark~3.3]{winter2011strongly} $D$ is $K_1$ injective hence  \cite[Theorem~2.2]{dadarlat2009kk} implies that in $D\otimes D$ there is a continuous path of unitaries $u_t$, for $0\leq t<1$ such that $u_0=1$ and $\lim_{t\to 1}  u_t(a\otimes b) u_t^*=b\otimes a$ for all $a$ and $b$ in $D$. For $\varepsilon>0$, fix $t'<1$ large enough to have  
$\|u_{t'}\Theta_m (d)u_{t'}^*-\Theta_n(d)\|<\varepsilon$ for all $d\in F$. Then fix 
$k(F,\varepsilon)$ large enough to have $0=t(0)<t(1)<\dots < t(k(F,\varepsilon)-1)=t'$ which satisfy $\|u_{t(j)}-u_{t(j+1)}\|<\varepsilon$ for all $j<k(F,\varepsilon)-2$ and $\|u_{k(F,\varepsilon)-1}\Theta_m (d)u_{k(F,\varepsilon)-1}^*-\Theta_n(d)\|<\varepsilon$ for all $d\in F$.  Transferring the unitaries $u_{t(j)}$ to $A$, by identifying $D\otimes D$ with $D_m\otimes D_n$ as in the proof of Theorem~\ref{T.Ainfty}, we obtain  $w^{m,n,\varepsilon}_j(F):=u_{t(j)}$ as required. 
 \end{proof}

Part \eqref{1.1.diagonalization} of the following lemma is closely related to  \cite[Lemma 9.6.7]{Fa:STCstar}, \cite[Theorem~3.1]{Ell:Derivations}, and \cite[Lemma~2.6]{mckenney2018forcing} (see also the introduction to the present section for a discussion of the origins of this lemma). Although the latter result is fairly general, it does not literally imply the statement that we need and we sketch a proof for reader's convenience.

\begin{lemma} \label{L.diagonalization} Suppose that $B$ is a non-unital \cstar-algebra that has an approximate unit $(e_n)$ consisting of projections. 

\begin{enumerate}
\item \label{1.diagonalization}	If $a_n$, for $n\in \bbN$, belong to $\cM(B)$ then  there are an increasing sequence $j(n)$ in $\bbN$ and $a_n^0$, $a_n^1$, for $n\in \bbN$, in $\cM(B)$ with the following property. With the projections (letting $j(0)=0$ and $e_0=0$)  
\[
p_n:=e_{j(n+1)}-e_{j(n)}
\]
the following holds. 
\begin{enumerate}
\item \label{1.1.diagonalization} $(p_{2i}+p_{2i+1}) a_n^0 (p_{2j}+p_{2j+1})=0$ for all $n$ and $i\neq j$. 
\item \label{2.1.diagonalization} $(p_{2i+1}+p_{2i+2}) a_n^1 (p_{2j+1}+p_{2j+2})=0$ for all $n$ and $i\neq j$. 
\item \label{3.1.diagonalization} $a_n-a_n^0-a_n^1\in B$ for all $n$. 	
\end{enumerate}	
\item \label{2.diagonalization} If in addition $B=A\otimes \cK$, for $A$ unital and such that $A=\lim_n A_n$, then $p_n$ can be chosen of the form $1_A\otimes r_n$ for orthogonal projections $r_n\in \cK$, so that for some $f\colon \bbN\to \bbN$   and all $i$ in addition to \eqref{1.1.diagonalization}--\eqref{3.1.diagonalization} we have 
\begin{enumerate}
	\item \label{1.2.diagonalization} $(p_{2i}+p_{2i+1}) a_n^0 (p_{2i}+p_{2i+1})$
	belongs to  
	\[
	(p_{2i}+p_{2i+1}) (A_{f(i)}\otimes \cK) (p_{2i}+p_{2i+1}),
	\]
	\item \label{2.2.diagonalization} $(p_{2i+1}+p_{2i+2}) a_n^1 (p_{2i+1}+p_{2i+2})$
	belongs to 
	\[
	(p_{2i+1}+p_{2i}) (A_{f(i)}\otimes \cK) (p_{2i+1}+p_{2i+2}),
	\] 
	and
	\item  \label{3.2.diagonalization}  $\sum_{n\leq m} r_n$, for $m\in \bbN$, is an approximate unit in $\cK$. 
\end{enumerate} 
\end{enumerate}
\end{lemma}

\begin{proof} 
\eqref{1.diagonalization} 
	Find an increasing sequence $j(n)$, for $n\in \bbN$, such that $j(0):=0$ and for all $m\leq n$ we have 
	\begin{align}\label{eq.e.estimate}
		\begin{split}
	\|(1-e_{j(n+1)})a_m e_{j(n)}\|<2^{-n-2},\\
	\|(1-e_{j(n+1)})a_m^* e_{j(n)}\|<2^{-n-2}.
	\end{split}
	\end{align}
This is possible because both $a_m e_{j(n)}$ and $a_m^* e_{j(n)}$ belong to the ideal $B$. 
Figure~\ref{Fig.D[E]} adapted from \cite[p. 263]{Fa:STCstar} may be helpful in visualizing the construction to follow.  
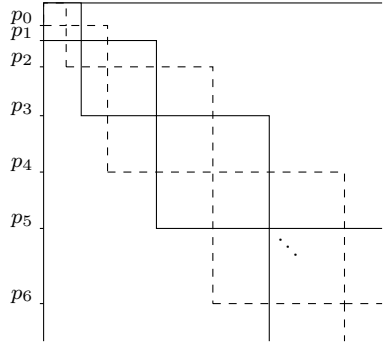
\begin{figure}[h]
	\begin{tikzpicture}[scale=.5, font=\scriptsize]
		\draw (0,0) -- (9,0);
		\draw (0,0) -- (0,-9);
		\draw[fill=white,draw=none] (0,0) rectangle (1,-1);
		\draw (0,0) rectangle (1,-1);
		\draw (1,-1) rectangle (3,-3);
		\draw (3,-3) rectangle (6,-6);
		\draw (6,-6) -- (6,-9);
		\draw (6,-6) -- (9,-6);
		\draw[dashed] (0,0) rectangle (.6,-.6);
		\draw[dashed] (.6,-.6) rectangle (1.7,-1.7);
		\draw[dashed] (1.7,-1.7) rectangle (4.5,-4.5);
		\draw[dashed] (4.5,-4.5) rectangle (8,-8);
		\draw[dashed] (8,-8) -- (8,-9);
		\draw[dashed] (8,-8) -- (9,-8);
		\foreach \x / \y in {0/.6,1/1,2/1.7,3/3,4/4.5,5/6, 6/8}
		{
			\node[anchor=east] at (0,-\y+.2) {$p_{\x}$};
			\draw (-.1,-\y) -- (0,-\y);
		}
		\foreach \z in {.3,.5,.7}
		\node at (6+\z,-6-\z) {$\cdot$};
	\end{tikzpicture}
	\caption{\label{Fig.D[E]} The supports of $a^0_n$ (with respect to any orthonormal basis diagonalizing all $p_n$) are included in the solid square regions and the supports of $a^1_n$ are included in the dashed line regions.}
\end{figure}

	Let $p_n:=e_{j(n+1)}-e_{j(n)}$ and let (the series on the right-hand side of each of the following formulas is  strictly convergent because $p_n$ are orthogonal projections in $B$)
		\begin{align*}
	a_m^0&:=\sum_n (p_{2n} a_m p_{2n} + p_{2n+1}a_m p_{2n} + p_{2n} a_m p_{2n+1}), \\
	a_m^1&:=\sum_n (p_{2n+1} a_m p_{2n+1} + p_{2n+2}a_m p_{2n+1} + p_{2n+1} a_m p_{2n+2}). 
	\end{align*}
Clearly the conditions \eqref{1.1.diagonalization} and \eqref{2.1.diagonalization} are satisfied. 
In order to verify \eqref{3.1.diagonalization},  fix $2k>m$. Then 
\(
a^0_m p_{2k}=(p_{2k}+p_{2k+1})a_m p_{2k}
\)
and
\(
a^1_m p_{2k} = p_{2k-1} a_m p_{2k}\).
Therefore 
\begin{align*}
	(a-a^0_m-a^1_m)p_{2k}
	&=\left(\sum_{i<2k-1} p_i+\sum_{i>2k+1} p_i\right)a_m p_{2k}\\
	&=(e_{j(2k-1)}+(1-e_{j(2k+2)}))a_m p_{2k}. 
\end{align*}
Similarly, $(a-a^0_m-a^1_m)p_{2k+1}
	=(e_{j(2k)}+(1-e_{j(2k+3)}))a_m p_{2k+1}$ for $2k+1>m$, and therefore 
	 $(a-a^0_m-a^1_m)p_{k}=
	 (e_{j(k-1)}+(1-e_{j(k+2)}))a_m p_{k}$ for all $k>m$. 
	 Since $p_{k}=e_{j(k+1)}-e_{j(k)}$, for $k>m$ we have, using \eqref{eq.e.estimate},
\begin{align*}
\|(a-a^0_m-a^1_m)p_{k}\|
&=\|(e_{j(k-1)}+(1-e_{j(k+2)}))a_m p_{k}\|\\
&\leq \|e_{j(k-1)} a_m (1-e_{j(k)})\|+\|(1-e_{j(k+2)}) a_m e_{j(k+1)}\|\\
&<2^{-k-2}+2^{-k-1}<2^{-k}.
\end{align*}
 Therefore all $n>m$ satisfy
\[
\|(a_m-a^0_m-a^1_m)(1-e_{j(n)})\|\leq \sum_{k\geq n} 
\|(a_m-a^0_m-a^1_m)p_k\|<\sum_{k=n}^\infty 2^{-k}=2^{-n+1}. 
\]
Since $n$ was arbitrary, we have  $a_m-a_m^0-a_m^1\in B$ for all $m$, and  \eqref{3.1.diagonalization} follows. 
	
	For \eqref{2.diagonalization}, assume $B=A\otimes \cK$ for a unital \cstar-algebra $A$. Apply \eqref{1.diagonalization} with $e_n=1_A\otimes q_n$, for an increasing approximate unit $q_n$ of $\cK(H)$ consisting of projections to produce first approximations to  $a^0_n$, $a^1_n$, and $p_n$ satisfying \eqref{1.1.diagonalization}--\eqref{3.1.diagonalization}. By the choice of the $e_n$’s, each $p_n$ is of the form $1_A\otimes r_n$ for a projection $r_n$ in $\cK(H)$ such that finite partial sums of $r_n$’s form an approximate identity in $\cK(H)$, which takes care of \eqref{3.2.diagonalization}.

Before the notation gets out of control, we define the \emph{cutdown of $x$ by a projection $q$} to be $qxq$. By replacing $a^l_n$ with $(1-e_{j(n)})a^l_n(1-e_{j(n)})$ for $l=0,1$ (note that the image of $a^l_n$ under the quotient map is unaffected by this change), we may assume that  $p_i a^l_n =a^l_n p_i=0$ for all $i<j(n)$ and $l=0,1$.   Then  in each block of the form $(p_{2i}+p_{2i+1})(A\otimes \cK)(p_{2i}+p_{2i+1})$ only finitely many of the $a_n^0$ have nonzero cutdowns, and in each block of the form $(p_{2i+1}+p_{2i+2})(A\otimes \cK)(p_{2i+1}+p_{2i+2})$ only finitely many of the $a_n^1$ have nonzero cutdowns. Since $A=\lim_l A_l$, the cutdown of each $a^0_n$ is within $1/(i+1)$ from some  
	  \[
	  b(n,i)\in (p_{2i}+p_{2i+1})(A_l\otimes \cK) (p_{2i}+p_{2i+1})
	  \]
	   for a large enough~$l$. The norm of $b(n,i)$ can be assumed to be no greater than $\|a^0_n\|$. 
	   	   Similarly, for every $i\geq j(n)$ there are $l\in \bbN$ and 
	   	   \[
	   	   c(n,i)\in (p_{2i+1}+p_{2i+2}) (A_l\otimes \cK) (p_{2i+1}+p_{2i+2})
	   	   \] 
	   	   which is within $1/(i+1)$ of the cutdown of $a^1_n$ by $p_{2i+1}+p_{2i+2}$ and satisfies $\|c(n,i)\|\leq \|a^1_n\|$. 
	   
	    Let $f(i)$ be the maximum of the $l$'s obtained in this way for all $n$ such that $i\geq j(n)$. For $i<j(n)$ let $b(n,i)=0$
	    and $c(n,i)=0$ (to be precise, this is $0$ in the corresponding cutdown of $A_{f(i)}\otimes \cK$). Thus in each block of the form $(p_{2i}+p_{2i+1})(A\otimes \cK)(p_{2i}+p_{2i+1})$,  every $a^0_n$ has a $1/(i+1)$-approximation  $b(n,i)$ in the corresponding cutdown of $A_{f(i)}\otimes \cK$ and in each block of the form $(p_{2i+1}+p_{2i+2}) (A_j\otimes \cK) (p_{2i+1}+p_{2i+2})$ every $a^1_n$ has a $1/(i+1)$-approximation $c(n,i)$ in the corresponding cutdown of $A_{f(i)}\otimes \cK$.  

A special case of \cite[Lemma~13.1.10]{Fa:STCstar} when $\{e_n\}$ is an increasing sequence of projections implies that for every $n$ the series   $b_n:=\sum_i b(n,i)$ is strictly convergent and therefore $b_n\in \cM(A\otimes \cK)$. Since $b(n,i)$ is within $1/i$ of the cutdown of $a^0_n$ by $p_{2i}+p_{2i+1}$,  we have  $\|(a^0_n-b_n)(1-e_{j(2i)})\|\leq 1/(i+1)$ for all $i$, and therefore  $a^0_n-b_n\in A\otimes \cK$. Moreover,~with $a^0_n$ replaced with $b_n$,  \eqref{1.1.diagonalization} continues to hold, and \eqref{1.2.diagonalization} holds by the choice of the function~$f$.

Similarly,  $c_n:=\sum_i c(n,i)$ defines an element of $\cM(A\otimes \cK)$ such that  $a^1_n-c_n\in A\otimes \cK$ and $c_n$ satisfies \eqref{2.1.diagonalization} and \eqref{2.2.diagonalization}, completing the proof.  
\end{proof}

A notational convention is in order. Suppose that $p_n$ is a sequence of orthogonal projections in $A$ such that $\sum_n p_n=1_{\cM(A\otimes \cK)}$ (such series is convergent in the  strict topology by \cite[Lemma~13.1.11]{Fa:STCstar}). Suppose in addition that $p_n=1_A\otimes r_n$, for $r_n\in \cK$. Then $p_n (A\otimes \cK) p_n$ is isomorphic to $M_m(A)$ (where $m=\rank(r_n)$, but the exact value of $m$ is of no importance for us) and we let 
\begin{equation}\label{Eq.Upsilon}
\Upsilon_n \colon A\to p_n (A\otimes\cK) p_n
\end{equation}
 be the diagonal embedding, $a\mapsto a\otimes 1_m$.

\begin{proof}[Proof of Theorem~\ref{T.Q(A)}] Suppose that $A$ and $D$ are separable and unital, $D$ is strongly self-absorbing, and $A$ is tensorially $D$-absorbing. 
Fix a separable \cstar-subalgebra $B$ of $\cQ(A\otimes \cK)$, and let $\{a_n\mid n\in \bbN\}$ be a subset of $\cM(A\otimes \cK)$ such that $\pi(a_n)$, for $n\in \bbN$, is dense in $B$. Write $A=\lim_n A_n$ as in the paragraph preceding Lemma~\ref{L.discretization}, so that $A_{n+1}=A_n\otimes D_n$ with $D_n\cong D$ for all $n$.  By Lemma~\ref{L.diagonalization} \eqref{2.diagonalization}, there are projections $p_i$, for $i\in \bbN$, in $A\otimes \cK$ such that $\sum_i p_i=1_{\cM(A\otimes \cK)}$ (the sum is convergent in the strict topology), and there are  $f\colon \bbN\to \bbN$ and $a_n^0$, $a_n^1$ in $\cM(A\otimes \cK)$ such that the following holds for all $n$ and all $i$ (all infinite sums are strictly convergent).   
\begin{enumerate}
	\item $a_n-(a_n^0+a_n^1)\in \cM(A\otimes \cK)$.
	\item $a_n^0=\sum_i (p_{2i}+p_{2i+1})a_n^0 (p_{2i}+p_{2i+1})$. 
	\item $a_n^1=\sum_i (p_{2i+1}+p_{2i+2})a_n^1 (p_{2i+1}+p_{2i+2})$. 
	\item $(p_{2i}+p_{2i+1})a_n^0 (p_{2i}+p_{2i+1})$ belongs to 
	\[
	(p_{2i}+p_{2i+1})( A_{f(i)}\otimes\cK)(p_{2i}+p_{2i+1})
	\] 
	\item 	$(p_{2i+1}+p_{2i+2})a_n^1 (p_{2i+1}+p_{2i+2})$ belongs to 
	\[
(p_{2i+1}+p_{2i+2})(A_{f(i)}\otimes\cK)(p_{2i+1}+p_{2i+2}).
	\] 
\end{enumerate}
 Let $F_i$, for $i\in \bbN$, be an increasing sequence of finite subsets of the unit ball of $D$ with dense union. Let $\varepsilon(j)=1/(j+1)$. Let $l(0)=0$ and (with $k(F_j,\varepsilon(j))$ as provided by Lemma~\ref{L.discretization}) let  $l(j+1):=l(j)+k(F_j,\varepsilon(j))$ for $j\geq 0$. Also let $f^+(j):=f(l(j))$ for $j\in \bbN$. 
 Define $\Psi\colon D\to \cM(A\otimes \cK)$ as follows.

 Fix $d\in D$. We will have $\Psi(d)=\sum_n x_n p_n$, for some $x_n$ in $p_n (A\otimes \cK) p_n$ to be determined shortly. Let $\Upsilon_n \colon A\to p_n (A\otimes\cK) p_n$ be the diagonal  embedding defined in \eqref{Eq.Upsilon}.  Thus for every $j\in \bbN$ the embedding $\Theta_j$ of $D$ into $A$ defined by identifying it with $D_j$ (the $j$-th copy of $D$ in $D^{\otimes\bbN}$)  used earlier determines a unital embedding 
\[
 \Upsilon_n\circ \Theta_j\colon D\to p_n (A\otimes \cK) p_n. 
 \]
We will need to twist this embedding a little. For $i\in \bbN$ and $j<k(F_i,\varepsilon(i))$ let, using $k(F_i,\varepsilon(i))$ and  the unitaries $w^{m,n,\varepsilon(i)}_j(F_i)$ produced by Lemma~\ref{L.discretization},
 \begin{equation}\label{eq.v.l(i)}
v_{l(i)+j}:= w^{f^+(i),f^+(i+1),\varepsilon(i)}_j(F_i). 
\end{equation}
 Since  every natural number can be uniquely represented in the form $l(i)+j$ for some $i$ and  $j<k(F_i,\varepsilon(i))$, this defines $v_n$ for all $n\in \bbN$ (analogous remark applies to the definition of~$\Psi_n$ below).
Then for all $i$ and $j<k(F_i,\varepsilon(i))-1$ we have 
\begin{equation}\label{Eq.v}
	\|v_{l(i)+j}-v_{l(i)+j+1}\|<\varepsilon(i).
\end{equation}
 Also, by \eqref{4.discretization} of Lemma~\ref{L.discretization},  for all $i$ we have 
\begin{equation}\label{Eq.Fi}
\|(\Ad v_{l(i)+k(F_i,\varepsilon(i))-1}) \circ  \Theta_{f^+(i)}(d)-\Theta_{f^+(i+1)}(d)\|<\varepsilon(i) \text{ for all } d\in F_i.
\end{equation}
For $n$ in $\bbN$, with $n=l(i)+j$ for unique $i$  and $j<k(F_i,\varepsilon(i))$, 
\begin{equation}
	\label{eq.Psi}
\Psi_{l(i)+j}(d):=(\Ad v_{l(i)+j}) \circ \Theta_{f^+(i)}(d)
\end{equation}
defines a $^*$-homomorphism $\Psi_n\colon D\to A$ for all $n$.

\begin{claim}\label{Eq.Psi-n-n+1} 
	\begin{enumerate}
		\item For all $i$ and $j<k(F_i,\varepsilon(i))$, the range of $\Psi_{l(i)+j}$ is included in $A\cap A_{f(i)}’$. 
	\item 	For all $n$ we have 
	$\|\Psi_n(d)-\Psi_{n+1}(d)\|\leq 2\varepsilon(i)$. 
	\end{enumerate}
\end{claim}

\begin{proof}The first part follows immediately from the fact that both the range of  $\Theta_{f^+(i)}$ and all $v_{l(i)+j}$ are included in $A\cap A_{f^+(i)}’$.  

Fix $n$. By \eqref{eq.v.l(i)}, if $n=l(i)+j$ for some $j$ such that $0\leq j< k(F_i,\varepsilon(i))-1$, then  $\|v_{l(i)+j}-v_{l(i)+j+1}\|<\varepsilon(i)$ and therefore  $\|\Psi_n(d)-\Psi_{n+1}(d)\|<2\varepsilon(i)$ for all $d$. If $n=l(i)+k(F_i,\varepsilon(i))-1$, then  $v_{n+1}=v_{l(i+1)}=1$ and $v_n=v_{l(i)+k(F_i,\varepsilon(i))-1}=w^{f^+(i),f^+(i+1),\varepsilon(i)}_{k(F_i,\varepsilon(i)-1}(F_i)$ satisfies, by \eqref{4.discretization} of Lemma~\ref{L.discretization}, 
\[
\|(\Ad v_n\circ  \Theta_{f^+(i)})(d)-\Theta_{f^+(i+1)}(d)\|<\varepsilon(i),
\]
and the left-hand side of this formula is equal to $\|\Psi_n(d)-\Psi_{n+1}(d)\|$ (see~\eqref{eq.Psi}), giving the required estimate. 
\end{proof}

Let (the right-hand side is again strictly convergent for all~$d$)
\[
\Psi(d):=\sum_n \Upsilon_n\circ \Psi_n(d).
\]
This is a $^*$-homomorphism of $D$ into $\cM(A\otimes \cK)$ such that 
\[
\Psi(1_D)=\sum_n p_n=1_{\cM(A\otimes\cK)}.
\]
 Since each $\Psi_n$ is an isometry we have $\Psi[D]\cap (A\otimes\cK)=\{0\}$, and therefore $\pi\circ \Psi$ defines an injective, unital $^*$-homomorphism of $D$ into $\cQ(A\otimes \cK)$.

\begin{claim} \label{C.Claim}
	 For every $n=l(i)+j$ with $j<k(\varepsilon(i))$ and  $d\in F_i$ we have  
\begin{enumerate}
\item \label{1.Claim}
$\Psi(d) p_{n}\in 
p_{n} (A\otimes \cK) p_{n}\cap (p_{n} (A_{f^+(i)}\otimes\cK)p_{n})'$.

	\item \label{2.Claim} If  $a\in p_{n+1} (A_{f^+(i)}\otimes \cK) p_{n}$ then 
	\[
	\|a (\Upsilon_n \circ \Psi_n(d))-(\Upsilon_{n+1}\circ \Psi_{n+1}(d))a\|<2\varepsilon(i)\|a\|.
	\] 
	\item \label{3.Claim} For all $m\in \bbN$ and $d\in D$, $[a_m,\Psi(d)]\in A\otimes \cK$. 
\end{enumerate} 
\end{claim}

\begin{proof} \eqref{1.Claim} follows from the first part of Claim~\ref{Eq.Psi-n-n+1}  and the fact that $\Upsilon_n$ is an isomorphic embedding of $A$ into a corner of $A\otimes \cK$. 

\eqref{2.Claim} Considering $a$ as a matrix over $A_{f^+(i)}$, by the choice of $v_n$ and $\Theta_{f^+(i)}$  all of its entries commute with $\Psi_n(d)=(\Ad v_n )\circ \Theta_{f^+(i)}(d)$. Therefore 
\[
a  (\Upsilon_n\circ \Psi_n(d))=(\Upsilon_{n+1}\circ\Psi_n(d))a. 
\]
The second part of Claim~\ref{Eq.Psi-n-n+1} implies that $\|\Psi_{n}(d)-\Psi_{n+1}(d)\|<2\varepsilon(i)$,  therefore $\|(\Upsilon_{n+1}\circ \Psi_{n+1}(d))a-(\Upsilon_{n+1}\circ \Psi_{n}(d))a\|<2\varepsilon(i)$, and the claim follows. 

\eqref{3.Claim}
Since $a_m-a^0_m-a^1_m\in A\otimes \cK$, it suffices to prove that for all $m$ and $d\in D$, both $[a_m^0,\Psi(d)]$ and $[a_m^1,\Psi(d)]$ belong to $A\otimes \cK$. By density it suffices to prove this for all $d\in \bigcup_i F_i$. Fix $i$, $d\in F_i$, and  $m$ and consider the commutator $[a_m^0, \Psi(d)]$.

 We have $\Psi(d)=\sum_n p_n \Psi(d) p_n$ and $a^0_m=\sum_n (p_{2n}+p_{2n+1})a^0_m (p_{2n}+p_{2n+1})$. Thus for all $n$ and $n’$ such that $|n-n’|\geq 2$ we have  $p_n [a^0_m, \Psi(d)]p_{n’}=0$.    By  (1), $p_n \Psi(d) p_n$ and $p_n a^0_m p_n$ commute for all $n$, hence $p_n [a^0_m, \Psi(d)]p_{n}=0$ for all $n$. Therefore, after rearranging the strictly convergent sum we  have  that $[a^0_m,\Psi(d)]$ is equal to the following. 
\begin{multline}\label{Eq.commutator}
\sum_n (p_{2n+1} a_m^0 p_{2n} \Psi(d)- \Psi(d)p_{2n+1} a_m^0 p_{2n})\\ + 
\sum_n (p_{2n} a_m^0 p_{2n+1} \Psi(d) - \Psi(d)p_{2n} a_m^0 p_{2n+1}).  
\end{multline}
Then \eqref{2.Claim} implies that for $d\in F_i$ the $n$th summand of the first infinite sum has norm smaller than $2\varepsilon(i)\|a_m^0\|$ (for $i$ such that  $l(i)\leq 2n<l(i+1)$). Since $\varepsilon(i)\to 0$ and the projections $p_n$ are orthogonal, this implies that the first sum strictly converges to an element of $A\otimes \cK$ for all $d\in \bigcup_i F_i$, and therefore for all $d\in D$.  Analogous argument shows that the second sum in \eqref{Eq.commutator} also belongs to $A\otimes \cK$, and therefore that $[a_m^0,\Psi(d)]\in A\otimes \cK$.
 
  Analogous argument gives $[a_m^1,\Psi(d)]\in A\otimes \cK$ for all $m$ and all $d\in D$, concluding the proof. \end{proof}
  
  Since $B$ is generated by $\pi(a_m)$, for $m\in \bbN$,  the unital copy of $D$ that is the range of $\pi\circ \Psi$  is in the relative commutant of $B$, as required.  Since $B$ was an arbitrary separable \cstar-subalgebra of $\cQ(A\otimes \cK)$, this completes the proof. 
\end{proof}

\section{Other applications}

 It is well-known that all embeddings of a strongly self-absorbing \cstar-algebra $D$ into an ultrapower of a $D$-absorbing \cstar-algebra $A$ are unitarily equivalent (\cite[Theorem~7.2.2]{Ror:Classification},  \cite{effros1978c}). Analogous results hold for embeddings of~$D$ into the asymptotic  sequence algebra of $A$. Corollary~\ref{C1} below shows that this holds for  path algebras and coronas and it  appears to be new. In the original version of the present paper I only claimed approximate unitary equivalence in both (1) and (2). The proof provided clearly implied asymptotic unitary equivalence, but Jamie Gabe pointed out that one can even obtain unitary equivalence.\footnote{At first this may appear to be interesting because neither $A_{\text\faForward}$ nor $\cQ(A\otimes \cK)$ is likely to be fully (or even quantifier-free) countably saturated (see \cite[\S 15 and \S 16]{Fa:STCstar}).
 	 	However, as the referee kindly pointed out, the model-theoretic point of view is misleading because  this is proven by a standard reindexing argument similar to one used below in the proof of Corollary~\ref{C1}. 
 	  } 
   His argument is included with his kind permission.   

\begin{corollary}
	\label{C1} Suppose that $D$ is a separable, unital,   and strongly self-absorbing  \cstar-algebra. If $A$ is separable, unital, and tensorially $D$-absorbing then the following holds. 
 \begin{enumerate}
\item Every two unital embeddings of $D$ into  $A_{\text\faForward}$ are unitarily equivalent. 
\item Every two unital embeddings of $D$ into  $\cQ(A\otimes \cK)$ are unitarily equivalent. 
\end{enumerate}
\end{corollary}

The following lemma is a straightforward application of  reindexing in coronas of $\sigma$-unital \cstar-algebras of Phillips and Weaver,  known as `(un)folding’, 
 and it is the key point of the proof of Corollary~\ref{C1}. 

\begin{lemma}\label{L.folding} Suppose that $E$ is a $\sigma$-unital, non-unital \cstar-algebra and that~$D$ is a separable, unital \cstar-algebra. Two unital copies of $D$ in $\cQ(E)$ are asymptotically unitarily equivalent if and only if they are unitarily equivalent. 
\end{lemma}

\begin{proof} Only the direct implication requires a proof.  Let $C$ denote the \cstar-subalgebra of $\cQ(E)$ generated by two asymptotically unitarily equivalent unital copies of $D$. We have a continuous path of unitaries $u_t$, for $0\leq t<1$, in $\cQ(E)$ which implements the asymptotic unitary equivalence. Embed $\cQ(E)$ into $\cQ(E)_{\text\faForward}$ diagonally and consider the diagonal image of $C$. The path $u_t$ gives a unitary $u$ in $\cQ(E)_{\text\faForward}$ that implements an isomorphism between the two copies of $D$ comprising~$C$. We thus have an embedding of the \cstar-algebra generated by $C$ and $u$ into the path algebra $\cQ(E)_{\text\faForward}$ which factors through a $^*$-homomorphism from $C$ into $\cQ(E)$.  The unitary $u$ can be lifted to~$\cQ(E)$ by  \cite[Proposition~1.4]{PhWe:Calkin} and this lift implements the required isomorphism between the two copies of~$D$.  
\end{proof}

\begin{proof}[Proof of Corollary~\ref{C1}]  Let  $E$ be $A\otimes \cK$ or $C_b([0,\infty),A)$. Then $\cQ(E)$ is $A_{\text\faForward}$ or $\cQ(A\otimes \cK)$ and it is $D$-saturated by Theorem~\ref{T.Ainfty} or Theorem~\ref{T.Q(A)}. By Lemma~\ref{L.folding} we need to prove that every two unital copies of $D$ in $\cQ(E)$ are asymptotically unitarily equivalent. 
 By $D$-saturation it will suffice to prove that two commuting unital copies of $D$ in $\cQ(E)$ are asymptotically unitarily equivalent.  Since~$D$ is nuclear and simple, any two commuting copies of~$D$ generate an isomorphic  copy of $D\otimes D$. By \cite[Remark~3.3]{winter2011strongly} and \cite[Theorem~2.2]{dadarlat2009kk} the flip automorphism of $D\otimes D$ is strongly asymptotically inner, and the desired conclusion follows.  
\end{proof}

Here are Corollary~\ref{T.RC} and its proof (see also Conjecture~\ref{Conj.Q}). 

\begin{corollary}
	\label{T.commutant}
Suppose that $D$ is a separable, unital,   and strongly self-absorbing  \cstar-algebra. If $A$ is separable and unital \cstar-algebra then the following are equivalent. 
\begin{enumerate}
	\item \label{1.RC} $A\otimes D\cong A$. 
	\item \label{2.RC} $D$  embeds unitally into $A_{\text\faForward}\cap A'$. 
	\item  \label{3.RC} For every separable \cstar-subalgebra $S$ of $A_{\text\faForward}$, $D$  embeds unitally into  $A_{\text\faForward}\cap S'$. 
\end{enumerate}
\end{corollary}

\begin{proof} \eqref{1.RC} implies \eqref{3.RC} is Theorem~\ref{T.Ainfty}, and \eqref{3.RC} trivially implies \eqref{2.RC}. 

It remains to prove that  \eqref{2.RC} implies  \eqref{1.RC}. 	
Let $\cU$ be a nonprincipal ultrafilter on~$\bbN$. Then $C_b([0,\infty),A)\to \ell_\infty(A): f\mapsto f\rs \bbN$ is a quotient map that lifts a $^*$-homomorphism $\Psi\colon A_{\text\faForward} \to A_\cU$ (where $A_\cU$ is the ultrapower of $A$ associated with $\cU$). Clearly $\Psi$ sends the diagonal copy of $A$ in $A_{\text\faForward}$ to the diagonal copy of $A$ in $A_\cU$.  Since $D$ is nuclear and simple,~$\Psi$ is injective on~$D$ and we conclude that $D$ embeds into $A_\cU\cap A'$. By \cite[Theorem~7.2.2]{Ror:Classification}, $A$ absorbs~$D$ tensorially. 
\end{proof}

In \cite[Theorem~2.8]{FaHaRoTi:Relative} it was proven that, with $A$ and $D$ as in Corollary~\ref{C1}\footnote{A sufficient assumption on $D$ is that it has approximately inner half-flip.},  the Continuum Hypothesis implies that $A_\cU\cap D'$ is isomorphic to $A_\cU$ and $\ell_\infty(A)/c_0(A)\cap D'$ is isomorphic to $\ell_\infty(A)/c_0(A)$. The proof uses countable saturation (see \cite[\S 16]{Fa:STCstar}) of these quotient algebras. Since neither $\cQ(A\otimes \cK)$ nor $A_{\text\faForward}$ is likely to be countably saturated (for $A_{\text\faForward}$ see \cite[Exercise~16.8.36]{Fa:STCstar}, but beware the following typo: `degree-1 saturated' should be `saturated'; the present formulation clearly contradicts \cite[Theorem~15.1.5]{Fa:STCstar}), we have only the following weak analog of \cite[Theorem~2.8]{FaHaRoTi:Relative} (for the definition of an elementary submodel see \cite[Appendix D]{Fa:STCstar}). 

\begin{corollary}\label{C2} Suppose that $D$ is a separable, unital,   and strongly self-absorbing  \cstar-algebra. If $A$ is separable, unital, and tensorially $D$-absorbing then the following holds. 
	\begin{enumerate}
	\item 	For any unital embedding of $\Phi\colon D\to A_{\text\faForward}$, the relative commutant $A_{\text\faForward}\cap \Phi[D]'$ is an elementary submodel of $A_{\text\faForward}$. 
	\item 	For any unital embedding of $\Phi\colon D\to \cQ(A\otimes \cK)$, the relative commutant $\cQ(A\otimes \cK)\cap \Phi[D]'$ is an elementary submodel of $\cQ(A\otimes \cK)$. 
	\end{enumerate}
\end{corollary}

\begin{proof} The proofs are analogous to the proof of \cite[Theorem~2.8]{FaHaRoTi:Relative}, but using Lemma~1.5 of the same paper in place of Lemma~1.6. 
\end{proof}

The possibilities that $A_{\text\faForward}$ and $A_{\text\faForward}\cap \Phi[D]’$ are isomorphic or that $\cQ(A\otimes \cK)$ and $\cQ(A\otimes \cK)\cap \Phi[D]’$ are isomorphic remain open. 
By \cite[Corollary 15.1.6]{Fa:STCstar} both $\cQ(A\otimes \cK)\cap \Phi[D]'$ and $A_{\text\faForward}\cap \Phi[D]'$ are countably degree-1 saturated, however this does not suffice to prove that $A_{\text\faForward}$ and $A_{\text\faForward}\cap \Phi[D]’$ are isomorphic or that $\cQ(A\otimes \cK)$ and $\cQ(A\otimes \cK)\cap \Phi[D]’$ are isomorphic.  Notably, $\cQ(H)$ is not even countably homogeneous (\cite{farah2015calkin}).  This poses a serious obstruction to using model-theoretic methods for constructing automorphisms of $\cQ(H)$ whose restriction to some separable subalgebra is not implemented by a unitary. It is  presently not known whether such automorphism of $\cQ(H)$ can exist in any model of ZFC.  The logical complexity of the assertion that such automorphism exists is~$\Sigma^2_1$ and therefore standard set-theoretic arguments (including Woodin's~$\Sigma^2_1$ absoluteness theorem, see \cite{Lar:Stationary}, also \cite{Wo:Beyond}) suggest that if the existence of such automorphism of $\cQ(H)$ is relatively consistent with ZFC, then it follows from the Continuum Hypothesis.

\section{Concluding remarks} 

\label{S.concluding}

To put Theorem~\ref{T.A} in proper context, note that forcing axioms MA and $\OCAT$ (see \cite[\S 8.7]{Fa:STCstar}) together imply that $\cQ(H)$ is not isomorphic to the corona of any simple \cstar-algebra not stably isomorphic to the algebra $\cK$ of compact operators on a separable, infinite-dimensional Hilbert space. This is a consequence of  \cite[Theorem~B]{vignati2018rigidity} and the straightforward fact that in this situation there are no topologically trivial isomorphisms. The same axioms imply that the Calkin algebra has no outer automorphisms (\cite{Fa:All}, \cite[\S 17.2--17.8]{Fa:STCstar}) and even that each one of its endomorphisms is unitarily equivalent to the amplifications by a matrix algebra (\cite{vaccaro2019trivial}). However, the Continuum Hypothesis implies that the Calkin algebra has outer automorphisms (\cite{PhWe:Calkin}, see also \cite[\S 17.1]{Fa:STCstar}).

It is not known whether there are nonisomorphic separable, simple, \cstar-algebras with isomorphic coronas (see \cite[Question 3.19]{farah2022corona}). 

\begin{conjecture} \label{C.Calkin} The Calkin algebra is not isomorphic to the corona of any simple, separable, non-unital \cstar-algebra other than $\cK$. 
\end{conjecture}

A few remarks on obtaining a candidate $A$ for a counterexample to Conjecture~\ref{C.Calkin}. When $A$ is of the form $B\otimes \cK$ for a unital~$B$, then  by \cite[Theorem~3.2]{rordam1991ideals} or by  \cite[Definition~2.5 and Theorem~2.8]{lin1991simple} and \cite{lin2004simple}, it has to be purely infinite and simple in order  to have a simple and purely infinite corona. Notably, a relative of the first part of Lemma~\ref{L.diagonalization} (to be precise, \cite[Theorem~3.1]{Ell:Derivations}) is crucial in  analyzing ideals of multiplier algebras and coronas  (see e.g.,  the proof of \cite[Theorem~2.8]{rordam1991ideals}).

By the main result of  \cite{zhang1990property}, a purely infinite and simple \cstar-algebra is either unital or stable (i.e., tensorially $\cK$-absorbing). There are  examples of \cstar-algebras $A$  that  are neither unital nor stable but have simple coronas (see \cite[Remark~3.3]{rordam1991ideals}), but I'll concentrate on the purely infinite, simple case.

By a result of Kirchberg, if $B$ is nuclear in addition to being simple,  purely infinite, and separable, then it necessarily absorbs $\cO_\infty$ tensorially   (see \cite[\S 7.2]{Ror:Classification}), and therefore Corollary~\ref{C.D} implies that $\cQ(B\otimes \cK)$ is not isomorphic to $\cQ(H)$. The non-nuclear examples of purely infinite, simple, and separable \cstar-algebras from \cite[Theorem~4.3.8 and Theorem~4.3.11]{Phi:Classification} all absorb $\cO_\infty$ tensorially, and are therefore ruled out as well. In \cite[Example~4.8]{blackadar1992approximately} and \cite[Corollary~1.2]{dykema1998purely}, purely infinite and simple  \cstar-algebras have been constructed inside type III factors  (by Blackadar's method or by taking elementary submodels, see \cite[Proposition~7.3.1]{Muenster}). These examples are not approximately divisible, hence not $\cO_\infty$-absorbing and possibly not $\cZ$-absorbing. However,  they inherit trivial $K$-groups from the ambient factors. This rules them out as counterexamples to Conjecture~\ref{C.Calkin},  since standard methods (see the paragraph following Corollary~\ref{C.D}) imply that $B$ would have to have the same $K$-theory as the algebra of compact operators, namely $K_0(B)=\bbZ$ and $K_1(B)=0$. Possibly the reduced free products as in    \cite[Theorem~3.1]{dykema1998purely} can produce a purely infinite and simple, but not $\cZ$-stable, \cstar-algebra with the right $K$-theory, but here is another example (using the methods of \cite{Muenster}). 

Here is another example.  Let $B_0$ be a separable elementary submodel of~$\cQ(H)$, and let $B$ be a separable elementary submodel of $\cQ(B_0\otimes \cK)$. Then~$B$ is purely infinite and simple (because this property of \cstar-algebras is axiomatizable, see \cite[Theorem~2.5.1]{Muenster}),  $K_0(B)=K_1(B_0)=K_0(\cK)$, and $K_1(B)=K_0(B_0)=K_1(\cK)$, as required. By Theorem~\ref{T.B} combined with a clever use of the folding argument  (i.e., \cite[Proposition~1.4]{PhWe:Calkin} again) communicated to me by Jamie Gabe, Chris Schafhauser and Stuart White, $B_0$ is not $\cO_\infty$-absorbing.

It is however not at all clear how would one construct a nontrivial isomorphism between $\cQ(B\otimes \cK)$  and $\cQ(H)$ (or how to assure that  $\cQ(B\otimes \cK)$ has a $K$-theory reversing automorphism). What is clear is that such isomorphism cannot be constructed in ZFC  (by Vignati's result mentioned earlier) and that the Continuum Hypothesis could be helpful (see \cite[\S 11]{farah2022corona} on this  and \cite[\S 5 and \S 6]{farah2022corona} for more information on this sort of a rigidity question). 

Theorem~\ref{T.B} gives a rather crude description of a small part of the $\Ext$ group of Wassermann's \cstar-algebra $\cW(\Gamma)$. I haven't been able to prove anything new about the $\Ext$ of the reduced  group algebra of the free group, $\cstr(F_2)$, although it is well-known that its Ext is not a group (\cite{haagerup2005new}).  The techniques  developed in~\cite{belinschi2022strong} may be relevant to this problem. As pointed out by Nigel Higson, there is no known  example of a separable, unital \cstar-algebra $A$ such that $\Aut(\cQ(H))$ acts on $\Ext^w(A)$ nontrivially in some model on ZFC (see \cite[Question~3.18]{farah2022corona}; note that we use the weak Ext, allowing unitaries with nonzero Fredholm index, because otherwise letting $A$ be $M_2(\bbC)$ provides an example), and $\cW(\Gamma)$ provides an interesting test case.  Under forcing axioms this action is trivial, since  all automorphisms of $\cQ(H)$ are inner. 
%

The following conjectures that there is an analog of Corollary~\ref{T.RC} (and \cite[Theorem~7.2.2]{Ror:Classification}) for $\cQ(A\otimes \cK)$.

\begin{conjecture} \label{Conj.Q}Suppose that $A$ and $D$ are separable and unital  \cstar-algebras, and that $D$ is  strongly self-absorbing. Then the following are equivalent. 
\begin{enumerate}
	\item $A\otimes D\cong A$. 
	\item For every separable \cstar-subalgebra $S$ of $\cQ(A\otimes \cK)$ there is a unital embedding of $D$ into $\cQ(A\otimes\cK)\cap S'$. 
\end{enumerate}
\end{conjecture}

\bibliographystyle{plain}
\bibliography{ifmainbib}

\end{document}